\documentclass[a4paper,12pt]{article}
\usepackage{hyperref}
\usepackage[english]{babel}
\usepackage[T1]{fontenc}
\usepackage[utf8]{inputenc}
\usepackage{amsfonts}
\usepackage{amssymb}
\usepackage{amsmath}
\usepackage{graphicx}
\usepackage{cite}
\usepackage{epsfig}
\usepackage{psfrag}
\usepackage{tikz}
\usetikzlibrary{matrix,arrows,decorations.pathmorphing,positioning} 
\usetikzlibrary{decorations.markings, arrows.meta}
\usepackage{url}
\usepackage{amsthm}
\usepackage{enumerate}
\usepackage{tikz-cd}

\newtheorem{proposition}{Proposition}[section]
\newtheorem{definition}{Definition}[section] 
  
\newtheorem{corollary}{Corollary}[section] 
\newtheorem{remark}{Remark}[section]  
\newtheorem{example}{Example}[section]

\newcommand{\Z}{\mathbb{Z}}
\newcommand{\Zb}{\mathbf{Z}}
\newcommand{\Lam}{\Lambda}
\newcommand{\calG}{\mathcal{G}}
\newcommand{\Ab}{\mathbf{Ab}}
\newcommand{\Ob}{\mathrm{Ob}}
\newcommand{\Hom}{\mathrm{Hom}}
\newcommand{\Mor}{\mathrm{Mor}}

\newcommand{\Fun}{\mathbf{Fun}}

\newcommand{\AbFun}{\mathbf{Ab\text{-}Fun}}

\newcommand{\Lm}{\Lambda_m}
\newcommand{\Lone}{\Lambda}
\newcommand{\AL}{A(L)}

\newcommand{\Set}{\mathbf{Set}}

\newcommand{\Gint}[1]{\int_{#1}}  
\newcommand{\m}{\mathfrak{m}}

\date{}

\title{\bf Multivariate Quandles as Groupoid Invariants}   

\author{Xerxes D. Arsiwalla$^{1, }$\footnote{\url{x.d.arsiwalla@gmail.com}}   {}  {}  {}  {}   Louis H. Kauffman$^{2, }$\footnote{\url{loukau@gmail.com}}    \\     
{ }  \\ 
{\it \small $^{1}$Wolfram Institute for Computational Foundations of Science,  USA}   \\
{\it \small $^{2}$University of Illinois at Chicago, USA} 
}

\begin{document}   
\maketitle

\begin{abstract}

We introduce a new structural framework for constructing multivariate Alexander quandles based on a groupoid ${\cal G}$, composed of a disjoint union of delooping groupoids of deck groups. Compared to standard Alexander modules over group rings $\Z[\Z]$, the  multivariate version involves modules over groupoid rings $\Z[{\cal G}]$. A manifestly categorical description is given in terms of  $\Z$-Algebroids, whose functorial module theory yields an equivalent definition of multivariate Alexander modules. Using the category of elements construction, we define multivariate Alexander quandles from this groupoid framework. Three classes of multivariate quandle operations follow.  In contrast to earlier approaches, we abstract the algebraic structure from explicit link-dependence. As a consequence, we have a structural framework for: (i) enumerating new classes of Alexander quandles, and (ii) finding new coloring invariants of links. We show that maps between $\Z$-modules, facilitated via multivariate quandle operations, realize a quiver presentation of quandles, which makes manifest the oidification of Alexander modules. Oidified Alexander modules, based on groupoids, provide a constructive framework for composing  univariate quandles to obtain new multivariate ones. Finally, we comment on the possibility of new link invariants coming from groupoid-based formulations, including the fundamental groupoid.  

\end{abstract}

\vspace{2pc}
{\it Keywords}: Alexander Quandles, Deck Groups, Groupoids, $\Z$-Algebroids, Link Colorings.  
  
\clearpage

\tableofcontents

\section{Introduction}

Quandles are algebraic structures, originally introduced by Joyce  \cite{joyce1982} and Matveev \cite{matveev1982}, independently. They axiomatize the well-known Reidemeister moves of knots and links, and for that reason, are useful tools to probe structural aspects and  topological invariants of three-dimensional manifolds (see \cite{el2015quandles} for an excellent pedagogical exposition). They are defined as:  
\begin{definition}[ \cite{joyce1982, matveev1982} ]  
\label{qaxioms}
A quandle is a set $Q$ equipped with a binary  operation $*$ (and a right-inverse operation $\bar{*}$) satisfying the following three axioms (for any $a, b, c \in Q$): 
\begin{enumerate}  
\item \text{Idempotence:}\footnote{The related notion of automorphic sets \cite{Brieskorn1988} or racks \cite{fenn1992racks} is obtained by relaxing the idempotence axiom. These become relevant when discussing invariants of framed links.}  $$a * a = a$$    
\item \text{Right-invertibility:}  $$(a * b) \;\bar{*}\; b = a = (a \;\bar{*}\; b) * b $$   
\item \text{Right-self-distributivity:}  
 $$ (a * b) * c = (a * c) * (b * c) $$  
 \end{enumerate} 
 \end{definition}
The fundamental quandle $Q(K)$ of a knot $K$ is a complete invariant of
oriented classical knots up to reverse oriented mirror image, meaning that non-equivalent knots have non-isomorphic fundamental quandles.  
Among the most important examples of quandle algebras are \emph{Alexander quandles}, which are defined upon $\Z$-modules over the Laurent polynomial ring  $\Lam := \Z[t, t^{-1}]$:  
 \begin{definition}[  \cite{joyce1982, matveev1982} ]
Given a $\Lam$-module $M$, the Alexander quandle over $M$ is defined by the operation 
\[
  a * b = t \, a + (1  -  t) \, b  
\] 
where $t$ is an invertible element of the $\Lam$-module (in the case of knots, $t$ is the generator of deck transformations of the infinite cyclic cover of the knot  complement). These are also referred to as \emph{affine Alexander quandles}.
 \end{definition}
For quandles of finite order, the module $M$ is expressed as a quotient  $\Z_n[t, t^{-1}]/(p(t))$, where  $p(t)$ is a non-zero polynomial such that $t$ is  invertible in the quotient.  For instance,  when $p(t) = t - k$, for some integer $k$ coprime to $n$, one obtains an    
Alexander quandle of order $n$, with quandle operation
$  a * b = k \, a + (1 - k) \, b \;  \pmod{n} $.      
More generally, with polynomials $p(t)$ of degree $m$, one obtains Alexander quandles of order $n^m$, where the parameter $t$ remains a variable in the module. For what follows here, we shall focus exclusively on quandles of finite order. 

The definition above refers to typical (also called standard) Alexander quandles, involving a single invertible element $t \in M$ (alternatively, the generator of the deck group).  More recently, in a series of papers  by Traldi  \cite{traldi0, traldi1, traldi2, traldi3, traldi4, traldi5, traldi6},  a notion of multivariate Alexander quandles for $m$-component links has been introduced. These involve multiple invertible elements $t_1, ... , t_m$, where each $t_i$ is associated to the $i^{th}$ component of the link $L$. The quandle operation in this case, takes the form 
\[  a * b = t_{\kappa (b)} \, a + (1  -  t_{\kappa (a)}) \, b    \]
where the component map $\kappa : A(D) \to \{ 1, ... , m  \}$ goes from the set of arcs $A$ of a link diagram $D$ to the index of the link component. Shown in \cite{traldi1}, this operation (along with the inverse) satisfies the quandle axioms in Definition~\ref{qaxioms}, and is a stronger link invariant than the single-variable Alexander quandle. It is associated to the link's module sequence  \cite{crowell1969elementary, hillman1981alexander}  with Alexander module $M_A (L)$  \cite{traldi1}.   

The data that determines Traldi's multivariate quandle operation  depends on the presentation of a given link, with the component map enabling the assignment of elements $t_1, ... , t_m$ to the arcs of the link diagram. Therefore, there is no generic way to directly combine the $t_1, ... , t_m$  with an abstract $\Lam_m$-module to get a canonical multivariate Alexander quandle independent of $L$. Traldi's formulation is associated to the link's module sequence \cite{crowell1969elementary, hillman1981alexander}. Consequently, comparing two quandles is directly tied to comparing two links, not two link-independent algebraic structures. Contrast this to the case of single-variable Alexander quandles, where the algebraic machinery is formulated independently of knot or link structures, and can be applied to compute colorability of a given knot or link \cite{kauffman2018colorings}.  Hence, we need a way to construct a multivariate version of Alexander quandles that is a priori link-independent. Such a construction would be useful for problems related to enumeration of  multivariate Alexander quandles (of any given finite order); as well as, for problems related to link colorability.  This is exactly the problem we address and solve in this paper. 
 
To resolve the aforementioned issue, one has to abstract the definition of  multivariate Alexander quandles in a way that is independent of the link presentation. Doing so, leads to a strict constraint  $t_{(a * b)} = t_a$ on the 
 assignment of the $t_a$'s, for all $a \in M$, elements of the module (see Section~2). It turns out that standard multivariable Alexander modules over the polynomial ring $\Z[t_1, ..., t_m, t_1^{-1}, ..., t_m^{-1}]$ do not, in general, satisfy this constraint. To solve precisely this problem, we introduce a new construction for (finite) multivariate Alexander quandles, which is based on groupoid rings\footnote{Note that some authors also refer to these as groupoid algebras over $k$, particularly when $k$ is a field \cite{khalkhali2013basic}. On the other hand, when $k$ is a unital commutative ring, the "groupoid ring" terminology is common \cite{lundstrom2006separable}. The latter applies to us here.}  and algebroids, rather than group rings.

The groupoid ${\cal G}$ we construct here consists of a disjoint union of delooping groupoids of deck groups. This is a totally disconnected groupoid. The objects of ${\cal G}$ label the $m$ components of a link, and the hom-sets $\Hom(i,i)$ are deck groups with generator $t_i$. Elements of this groupoid are thus "tagged" by a label referring to the vertex group to which they belong. We show that, compared to standard Alexander modules over the group ring $\Z[\Z]$, the multivariate case works much more naturally with modules over a groupoid ring $\Z[{\cal G}]$.  The "oidified" counterpart of  our groupoid ring is identified as a $\Z$-algebroid, $\Zb[{\cal G}]$  (these objects are also referred to as  preadditive categories, ringoids, or $Ab$-enriched categories)   \cite{Mitchell1972, Kelly1982, mitchell1985, avi2026, nlab:ab-enriched_category}.  

We present the module theory for the groupoid ring $\Z[{\cal G}]$, and the corresponding algebroid, $\Zb[{\cal G}]$. Turns out, these are two different presentations of the same module theory. This leads to a groupoid-based formulation of multivariate Alexander modules.  Multivariate Alexander quandles arise naturally from our groupoid framework, and satisfy the three   quandle axioms. We then find three classes of quandle operations following our formulation of multivariate Alexander quandles. One of them is similar to Traldi's operation \cite{traldi1}  (while being independent of any link presentation), whereas, the other two are new multivariate quandle operations. Furthermore, we show that maps between module components, facilitated via quandle multiplication, realize quiver presentations of multivariate Alexander quandles, thus depicting the oidification of Alexander modules, and the compositional structure of multivariate Alexander quandles from univariate ones. Given the computational framework we have developed here, we   explicitly compute new examples of multivariate Alexander quandles, adding to the existing literature classification  \cite{nelson1, nelson2, lopes2006finite, vendramin2012, clark2014quandle, higashitani2024classification}; and also demonstrate a computation of link colorability within our framework. 

This article is organized as follows:  we begin with preliminaries on Fox calculus applied to $m$-component links in Section 2; we then establish the groupoid structure for constructing new classes of Alexander modules in Section 3; then in Section 4, we present definitions of groupoid rings and $\Z$-algebroids, relevant to the definition of multivariate Alexander modules; in Section 5, we define multivariate Alexander quandles and derive new quandle operations; Section 6 introduces quiver presentations of quandles; Section 7 concerns computations, namely, enumeration of new multivariate Alexander quandles, and an application to link colorability; finally in Section 8, we conclude with closing remarks and future directions.

\section{Fox Calculus for $m$-Component Links}

Most of this section will consist of preliminary definitions, notation and results on Fox calculus \cite{fox1}, adapted to $m$-component links (following  \cite{crowellfox, hillman1981alexander, burde2013knots}). We also derive an operation for multivariate Alexander quandles, similar to what appeared in \cite{traldi1}. The informed reader may directly jump to that derivation.   

Following standard convention, $L = K_1 \cup \cdots \cup K_m \subset S^3$ denotes an oriented link of $m$ components, $X_L = S^3 \setminus N(L)$ its
exterior, and $\pi_L = \pi_1(X_L)$ the link group.
The case $m = 1$ recovers knot-theoretic results; all genuine
differences from that case are what we are interested in.

For a knot ($m=1$) one has $H_1(X_K;\Z) \cong \Z$, generated by a
single meridian $\mu$.
For an $m$-component link the situation is richer:
\[
  H_1(X_L;\,\Z) \;\cong\; \Z^m,
\]
freely generated by the meridians $\mu_1, \ldots, \mu_m$ of the $m$
components.
The abelianization homomorphism is
\[
  \phi \colon \pi_L \;\twoheadrightarrow\; \Z^m,
  \qquad
  \phi(\mu_i) = e_i,
\]
where $e_i$ is the $i$-th standard basis vector.
Each Wirtinger generator $x$ belongs to a definite component
$K_{\ell(x)}$, so
\[
  \phi(x) = t_{\ell(x)} \in \Z^m.
\]

Now, the integral group ring of $\Z^m$ is the ring of
\emph{multivariable Laurent polynomials}:
\[
  \Lm \;=\; \Z[\Z^m]
       \;=\; \Z\bigl[t_1^{\pm 1}, \ldots, t_m^{\pm 1}\bigr],
\]
one variable $t_i$ per component.
This replaces the single-variable ring
$\Lone = \Z[t^{\pm 1}]$ of the knot case.

To define the Alexander module, one works with a covering space of the link complement.  Let $\widetilde{X}_L \to X_L$ be the universal abelian cover 
associated to the abelianization map $\phi$, with deck group $\Z^m$, such that each generator $t_i$ acts as a deck transformation corresponding to $\mu_i$.

\begin{definition}[Alexander Module of a Link \cite{alexander1928} (see \cite{burde2013knots} for a pedagogical exposition)]  
The \emph{Alexander module} of $L$ is then defined to be  
\[
  \AL \;=\; H_1\!\bigl(\widetilde{X}_L;\,\Z\bigr)
\]
as a $\Lm$-module.  
\end{definition}

Next, the standard way to formulate Fox calculus makes use of the Wirtinger presentation of the link group. 
Let $D$ be an oriented diagram of $L$ with $n$ crossings.
Label the arcs $a_1, \ldots, a_n$ and assign to each arc $a_i$ a
\emph{component label} $\ell(i) \in \{1, \ldots, m\}$ indicating
which component of $L$ it belongs to.
The Wirtinger generator $x_i$ corresponding to $a_i$ satisfies
$\phi(x_i) = t_{\ell(i)}$.  The \emph{Wirtinger presentation} of the link group  $\pi_L$ is expressed in terms of $n$ generators and $n$ relations (though only $n-m$ of these relations are independent) as follows:   
\[
  \pi_L \;=\;
  \langle\, x_1, \ldots, x_n \;\mid\; r_1, \ldots, r_n \,\rangle
\]
where at each crossing $k$, we will denote the overstrand $x_b$ passes over
understrands $x_a$ (entering) and $x_c$ (exiting), giving
\begin{eqnarray*}
  r_k \colon x_c &=& x_b x_a x_b^{-1}   \qquad  \mbox{for a positive crossing with sign $+1$}       \\
        r_k \colon x_c &=& x_b^{-1} x_a x_b  \qquad  \mbox{for a negative crossing with sign $-1$} 
\end{eqnarray*}
In uniform notation, we write $r_k \equiv x_b^{\varepsilon_k} x_a
x_b^{-\varepsilon_k} x_c^{-1} = 1$ with $\varepsilon_k \in \{+1,-1\}$.

Following \cite{fox1}, the Fox free differential calculus can now be expressed in terms of  $F_n = F(x_1, \ldots, x_n)$,  the free group over the $n$ generators,  and its integral group ring,  $\Z[F_n]$.  
The \emph{Fox derivatives} are  $\Z$-linear maps
\[
  \frac{\partial}{\partial x_i}
  \colon \Z[F_n] \longrightarrow \Z[F_n]
\]
defined by:
\[
  \frac{\partial 1}{\partial x_i} = 0,
  \qquad
  \frac{\partial x_j}{\partial x_i} = \delta_{ij},
  \qquad
  \frac{\partial x_j^{-1}}{\partial x_i} = -\delta_{ij}\,x_j^{-1},
  \qquad
  \frac{\partial(uv)}{\partial x_i}
    = \frac{\partial u}{\partial x_i} + u\,\frac{\partial v}{\partial x_i}.
\]
They satisfy the "fundamental formula" in \cite{fox1}: 
\[
  u - 1 \;=\; \sum_{i=1}^n \frac{\partial u}{\partial x_i}(x_i - 1)
  \qquad \text{for all } u \in \Z[F_n].
\]
The abelianization $\phi \colon \pi_L \to \Z^m$ extends to a ring
homomorphism $\phi \colon \Z[\pi_L] \to \Lm$ by
$\phi(x_i) = t_{\ell(i)}$. In contrast to the univariate case, here we assign a distinct $t_{\ell(i)}$ for each generator $x_i$.

Taking the Fox derivatives of the relations at positive crossings, and then  applying $\phi$,  we obtain:
\[
  \phi\!\left(\frac{\partial r_k}{\partial x_a}\right) = t_{\ell(b)},
  \qquad
  \phi\!\left(\frac{\partial r_k}{\partial x_b}\right)
    = 1 - t_{\ell(a)},
  \qquad
  \phi\!\left(\frac{\partial r_k}{\partial x_c}\right)
    = -1
\]

Likewise, at negative crossings we get:
\[
  \phi\!\left(\frac{\partial r_k}{\partial x_a}\right) = t_{\ell(b)}^{-1},
  \qquad
  \phi\!\left(\frac{\partial r_k}{\partial x_b}\right)
    = t_{\ell(b)}^{-1} \left(t_{\ell(a)} - 1\right),
  \qquad
  \phi\!\left(\frac{\partial r_k}{\partial x_c}\right)
    = -1.
\]
where we additionally have the condition
\[  t_{\ell(c)} = t_{\ell(a)}  \]
which follows from $r_k = 1$. 

In the above notation, the $n \times n$ matrix over $\Lm$ 
\[
  \mathcal{M}(t_1, \ldots, t_m)
  \;=\;
  \left(\phi\!\left(\frac{\partial r_i}{\partial x_j}
        \right)\right)_{1 \leq i,j \leq n}.
\]
is called the Alexander matrix of $L$.  The Alexander module is then presented by $\mathcal{M}(t_1,\ldots,t_m)$.  We will use this to read the crossing relations in $A(L)$.  Let $\bar{x}_1, \ldots, \bar{x}_n \in \AL$ denote the images of the  Wirtinger generators after abelianization and lifting to the
universal abelian cover.  
At a positive crossing with overstrand $x_b \in K_{\ell(b)}$ and
understrands $x_a, x_c \in K_{\ell(a)}$, the Fox calculus gives in
$\AL$:
\[   t_{\ell(b)}\,\bar{x}_a  +  (1 - t_{\ell(a)})\,\bar{x}_b  - \bar{x}_c  \;=\; 0 
\]
Solving for the outgoing arc $\bar{x}_c$:
\[
  \bar{x}_c 
  \;=\;
  t_{\ell(b)}\,\bar{x}_a + (1 - t_{\ell(a)})\,\bar{x}_b  \]
This defines a multivariate Alexander quandle operation for a link (later, we shall see that there also exists two other multivariate quandle operations). We then have:
\begin{definition} 
A \emph{multivariate Alexander quandle} is a $\Lm$-module $Q$
equipped with the following binary operation (using the standard notation $a \equiv \bar{x}_a$), applied to a positive crossing, where `a' refers to the understrand:   
\begin{eqnarray}
  a \;*\; b \;=\; t_b\,a + (1 - t_a)\,b  
\label{foxmqdef1}
\end{eqnarray}  
with the constraint  $t_{(a \,*\, b)} = t_a$.

The inverse operation, applied to negative crossings with `a' once again being the understrand, is given by: 
\begin{eqnarray}
   a \;\bar{*}\; b \;=\; t_b^{-1}\,a + t_b^{-1} \, (t_a - 1)\,b   
\label{foxmqdef2}
\end{eqnarray}  
with the corresponding constraint being  $t_{(a \,\bar{*}\, b)} = t_a$.
\end{definition}
It is easy to check that the above operation satisfies the three quandle axioms. 
Idempotence and right-invertibility are straightforward. Let us now check for self-distributivity: 
\begin{eqnarray*}  
\hspace{-5cm}  LHS &=& (a \,*\, b) \,*\, c   \\
 &=&  t_c \, ( t_b\,a + (1 - t_a)\,b  )  + (1 - t_a) \, c
\end{eqnarray*} 
\begin{eqnarray*}    
\hspace{1.8cm}   RHS &=& (a \,*\, c) \,*\,  (b \,*\, c)  \\
 &=&  t_b \, ( t_c\,a + (1 - t_a)\,c  )  + (1 - t_a) \, ( t_c\,b + (1 - t_b)\,c  ) 
\end{eqnarray*} 
where the $LHS$ matches the $RHS$. 

The operations appearing in eqs.~(\ref{foxmqdef1}) and (\ref{foxmqdef2}) are  analogous to what was reported in \cite{traldi1}  (Traldi's operation was tied to a  $\Lam_m$-linear epimorphism $\phi_L$, originating from the link's module sequence).  The constraints  $t_{(a * b)} = t_a$ and $t_{(a \,\bar{*}\, b)} = t_a$   (corresponding to the inverse operation)  serve as conditions on right ideals of the quandle algebra (see Proposition~\ref{prop5.2v0}), and thus, restrict the assignment of the multiple $t$-variables to specific module elements. In general, any multivariable Alexander module over  $\Z[t_1, ..., t_m, t_1^{-1}, ..., t_m^{-1}]$ will not be compatible with these constraints (following   Proposition~\ref{prop5.2v0}).  This calls for a new construction of multivariate Alexander modules which not only are compatible with these conditions, but also fully specify the assignment of $t$-variables to module elements. Below, we go on to show that our framework based on groupoid rings does just that.  

Along the way, we also find two other multivariate quandle operations. Unlike earlier approaches, our groupoid ring framework distinguishes the algebraic machinery from explicit link dependence in a natural way, thus paving the way for using multivariate Alexander quandles as computational gadgets in their own right.

\section{A Groupoid for Deck Groups }

As noted above (Section~2),  univariate (finite) Alexander modules, as $\Lm$-modules, are polynomial rings over cyclic groups, with a designated invertible element, which acts as the generator of deck transformations associated to the covering space of the knot complement. Given that, the question we wish to answer is: what suitable generalizations of these structures can one define such that they carry multivariate data, and are compatible with quandle axioms? Towards this goal, let us begin by examining various products of cyclic groups and their associated rings. 

We start by distinguishing between two main types of compositions on families of cyclic groups that yield fundamentally different algebraic structures. 
On one hand, compositions between finite abelian groups commonly involve  their direct sums, which preserves the group structure.  Given positive integers $n_1, \ldots, n_m$, the classical direct sum of cyclic groups is
\[
  \bigoplus_{i=1}^{m} \Z_{n_i}  
\]
with group operation defined coordinate-wise. 

In this work, we consider
an alternative construction, involving a disjoint union, where the elements are instead  ``tagged'' by indices, that is,  $(i, \Z_{n_i})$ denotes a copy of $\Z_{n_i}$ with index $i$:    
\[
  (1,\Z_{n_1}),\; (2,\Z_{n_2}),\; \ldots,\; (m,\Z_{n_m})  
\]
and examine its algebraic properties.

Let us demonstrate this construction with a simple example. 
\begin{example}
Recall that the direct sum $\Z_2 \oplus \Z_3$ consists of ordered pairs
with component-wise addition:
\[
  \Z_2 \oplus \Z_3 = \{(a,b) : a \in \Z_2,\; b \in \Z_3\}
\]
with elements
\[
  \{(0,0),\; (0,1),\; (0,2),\; (1,0),\; (1,1),\; (1,2)\}
\]
and group operation
$(a,b) + (a',b') = (a + a' \bmod 2,\; b + b' \bmod 3)$.
The total order is $2 \cdot 3 = 6$. 

In contrast, consider the disjoint union of $\Z_2$ and $\Z_3$ as sets,
where each element is tagged by an index indicating its component of
origin:    
\[
  (1,\Z_2) \bigsqcup \, (2,\Z_3) = \{(1,0),\;(1,1),\;(2,0),\;(2,1),\;(2,2)\} 
\]
where $(i, a)$ denotes a "tagged" element of the set, with $i$, an index labelling an algebra, and $a$, an element of the algebra labelled by $i$.  The disjoint union of $\Z_2$ and $\Z_3$ above has total cardinality $2 + 3 = 5$.

\end{example}

More generally, the direct sum of $m$ cyclic groups, written as     
$\bigoplus_{i=1}^{m} \Z_{n_i}$, has order $\prod_{i=1}^{m} n_i$.  Whereas, the disjoint union, denoted by  $\bigsqcup_{i=1}^{m} \Z_{n_i}$, has cardinality $\sum_{i=1}^{m} n_i$.  The direct sum of groups carries a canonical group structure inherited from the components. Whereas,  the disjoint union of groups  does not inherit a group structure. Unlike a direct sum, whose elements are tuples $(a_1,\ldots,a_m)$, the disjoint union has tagged pairs $(i,a)$.  
 
Let us write a definition for the tagged-element construction on cyclic groups:   
\begin{definition}
The \emph{disjoint union} (or tagged-element construction) on cyclic groups
$\Z_{n_1}, \ldots, \Z_{n_m}$ is the set
\[
  {\cal X} = \bigsqcup_{i=1}^{m} \Z_{n_i} = \{(i,a) : 1 \le i \le m,\; a \in \Z_{n_i}\}
\]
equipped with the partial binary operation $+_{\cal X}$ defined by
\[
  (i,a) +_{\cal X} (i,b) = (i,\; a + b \bmod n_i)
\]
for elements in the same component, and undefined for elements $(i,a)$
and $(j,b)$ with $i \neq j$.
\end{definition}
The partial operation $+_{\cal X}$ is precisely what obstructs a global group structure, and instead makes ${\cal X}$ into a groupoid. For what follows here, it is this groupoid structure, that will support a module theory accommodating  multivariate Alexander quandles.

The groupoid associated to the above disjoint union of tagged elements has a natural category-theoretic definition. In a category, some compositions are defined, whereas others are not, just as in our discussion for the groupoid structure above.  Groupoids were originally introduced by H. Brandt in \cite{brandt1927verallgemeinerung};  modern treatments can be found in \cite{brown1987groups, weinstein1996groupoids, brown2006topology, nlab:groupoid}.

\begin{definition}
\label{defgcygr}
Let $n_1, \ldots, n_m$ be positive integers. We define the \emph{groupoid}  $\calG(n_1, \ldots, n_m)$ as a category, whose objects are:  
\[ \Ob(\calG) = \{1, 2, \ldots, m\}  \]
with morphisms:  
 \[
          \Hom(i,j) =
          \begin{cases}
            \Z_{n_i} & \text{if } i = j, \\
            \varnothing & \text{if } i \neq j.
          \end{cases}
        \]
  For $a, b \in \Hom(i,i) = \Z_{n_i}$,  compositions are defined by
        \[
          b \circ a = a + b \pmod{n_i}.
        \]
 subject to associativity   
        \[
          (c \circ b) \circ a = c \circ (b \circ a) \pmod{n_i},
        \]
     for any $a, b, c \in \Z_{n_i}$.  
     
Furthermore, for each object $i$, there exists an identity morphism  $\mathrm{id}_i = 0 \in \Z_{n_i}$, such that  for any $a \in \Z_{n_i}$,
        \[
          a \circ \mathrm{id}_i =  a =  \mathrm{id}_i \circ a 
        \]

And,  for each $a \in \Hom(i,i) = \Z_{n_i}$, the inverse morphism is defined as 
        $a^{-1} = -a \pmod{n_i}$,  satisfying 
        \[
          a \circ a^{-1} = \mathrm{id}_i = a^{-1} \circ a
        \]    
\end{definition}

Note that the above groupoid $\calG(n_1, \ldots, n_m)$ is totally disconnected since
$\Hom(i,j) = \varnothing$ for $i \neq j$. The vertex group at object
$i$ is precisely $\Z_{n_i}$.

\begin{remark}  
The groupoid $\calG(n_1, \ldots, n_m)$ is a disjoint union of delooping   groupoids of cyclic groups (see \cite{nlab:groupoid}). More precisely, as an isomorphism of groupoids, we have:   
\[
  \calG(n_1, \ldots, n_m) \;\cong\; \bigsqcup_{i=1}^{m} B\Z_{n_i}  
\]
where $B\Z_{n_i}$ is the delooping groupoid of $\Z_{n_i}$ and
$\bigsqcup$ denotes the disjoint union (or the coproduct in the category of
groupoids).  Thus, $\bigsqcup_i B\Z_{n_i}$ has object  set $\{1, \ldots, m\}$ (one object from each $B\Z_{n_i}$, which we
label by $i$), and
\[
  \Hom_{\bigsqcup_i B\Z_{n_i}}(i,j) =
  \begin{cases}
    \Z_{n_i} & \text{if } i = j, \\
    \varnothing & \text{if } i \neq j,
  \end{cases}
\]
which coincides with the morphism sets of $\calG(n_1, \ldots, n_m)$.  
\end{remark}

Note that, the algebraic and categorical descriptions of the disjoint union construction easily translate into each other.  
For ${\cal X} = \bigsqcup_{i=1}^{m} \Z_{n_i}$ as the set of tagged elements of the  disjoint union, there is a canonical bijection
\[
  {\cal X} \longleftrightarrow \Mor(\calG) \; :=  \bigsqcup_{i,j \in \Ob(\calG)} \Hom(i,j)
\]
between ${\cal X}$ and the set of all morphisms in $\calG$. Given that  in our case $\calG$ is totally disconnected, the above bijection is realized as 
\[
{\cal X} =  \bigsqcup_{i=1}^{m} \Z_{n_i}  = \bigsqcup_{i=1}^{m} \Hom(i,i) = \Mor(\calG)  
\]
The bijection sends $(i,a) \in {\cal X}$ to the morphism $a \in \Hom(i,i)$.  Under this bijection, the partial operation $+_{\cal X}$ on tagged elements  corresponds precisely to composition of morphisms in 
$\calG$. The operation $(i,a) +_{\cal X} (i,b) = (i, a+b)$ corresponds to
$b \circ a = a + b$ in $\Hom(i,i)$, and the undefined status of
$(i,a) +_{\cal X} (j,b)$ for $i \neq j$ corresponds to the
non-composability of morphisms with mismatched source and target.  This  justifies the groupoid $\calG$ as the natural
algebraic structure on the disjoint union  construction, and it 
precisely captures the partial composability that distinguishes
disjoint unions from direct sums.

\begin{remark}
The groupoid $\calG(n_1, \ldots, n_m)$  we have constructed here, can be interpreted as a means to simultaneously encode $m$ independent deck groups such that the generator $t_i$ ($1 \leq i \leq m$) of the $i^{th}$ deck group is associated to the $i^{th}$ link component.   
\end{remark}

\section{Groupoid Rings, $\Z$-Algebroids \& Alexander Modules}

Having established the groupoid structure of $\calG = \calG(n_1, \ldots, n_m)$, constructed from the disjoint union of cyclic groups, we now present the corresponding ring / algebroid and module theory for this groupoid. What we need is an algebraic entity that generalizes classical group rings $\Z[G]$ and their  modules to the  groupoid setting - namely, groupoid rings. The latter are also referred to as groupoid algebras or groupoid convolution algebras  (see \cite{renault2006groupoid, weinstein1996groupoids, khalkhali2013basic, nlab:category_algebra}). Groupoid rings are rings with local units, whose idempotent elements implicitly encode the objects of the groupoid \cite{Wisbauer1991}. Nonetheless, these are rings, and not categories. 

In fact, the horizontal categorification or oidification of rings lead to   manifestly categorical structures - namely, algebroids (also referred to as ringoids, when the category is small). These  were originally introduced by B. Mitchell as "rings with several objects" \cite{Mitchell1972, mitchell1985},  and are also realized as enriched categories  \cite{Kelly1982, nlab:ab-enriched_category}. In what follows here, we shall discuss both, groupoid rings and algebroids, with respect to $\calG(n_1, \ldots, n_m)$. These definitions will provide us with the precise language for expressing multivariate Alexander modules and quandles.

\subsection{Groupoid Rings \& their Oidification}

Let us begin with the definition of the integral groupoid ring, which
generalizes the notion of an integral group ring.  More generally, groupoid rings are defined over a unital commutative ring $k$  \cite{weinstein1996groupoids,   dokuchaev2, lundstrom2006separable, machado2025non,   nlab:category_algebra}.  
When $k = \Z$, we have:   
\begin{definition}  
Let $\calG$ be a groupoid. The \emph{integral groupoid ring}
$\Z[\calG]$ is defined as follows:

As an abelian group 
\[
  \Z[\calG] = \bigoplus_{g \in \Mor(\calG)} \Z \cdot g,
\]
the free abelian group on the set of all morphisms of $\calG$. Elements
are formal finite $\Z$-linear combinations $\sum_g a_g \cdot g$ with
$a_g \in \Z$.

Multiplication is defined on basis elements by  
\[
  g \cdot h =
  \begin{cases}
    g \circ h & \text{if } \operatorname{sr}(g) = \operatorname{tr}(h), \\
    0         & \text{otherwise,}
  \end{cases}
\]
where $\operatorname{sr}, \operatorname{tr} : \Mor(\calG) \to \Ob(\calG)$ denote the source and target
maps, and extended to all of $\Z[\calG]$ by $\Z$-bilinearity.
\end{definition}

\begin{remark}
From the above definition, one can check that the multiplication on $\Z[\calG]$ is indeed associative. 
\end{remark}

In general, groupoid rings lack a global multiplicative identity (when the underlying object set is not finite).  Instead, one has a set of idempotent elements that serve as local units (see \cite{Wisbauer1991} for a thorough treatment of these rings).  For  $\Z[\calG]$, with $|\Ob(\calG)| > 1$, the local units are defined as follows: 

\begin{definition}[ (see \cite{Wisbauer1991}) ]  
\label{ludef}
The \emph{local units} of $\Z[\calG]$ are defined as:       
For each object $x \in \Ob(\calG)$, the identity morphism
$e_x := \mathrm{id}_x$ satisfies:
\begin{enumerate}  
  \item $e_x^2 = e_x$ (idempotence).
  \item $e_x \cdot g = g$ for all $g$ with $\operatorname{tr}(g) = x$.
  \item $g \cdot e_x = g$ for all $g$ with $\operatorname{sr}(g) = x$.
  \item $e_x \cdot e_y = 0$ for $x \neq y$ (orthogonality).
\end{enumerate}
The set $\{e_x : x \in \Ob(\calG)\}$ forms a complete system of
orthogonal local idempotents. Moreover, when $\Ob(\calG)$ is finite, the sum
$\sum_{x \in \Ob(\calG)} e_x$ acts as a two-sided identity on
$\Z[\calG]$.   
\end{definition}

We now apply the above definitions to our groupoid of interest, $\calG = \calG(n_1, \ldots, n_m)$ from Definition~\ref{defgcygr}.  

\begin{proposition}   
\label{prop-tdg1} 
For the totally disconnected groupoid $\calG = \calG(n_1, \ldots, n_m)$, defined in Definition~\ref{defgcygr},  the integral groupoid ring takes the form
\[
  \Z[\calG] = \bigoplus_{i=1}^{m} \Z[\Z_{n_i}]  
\]
as an abelian group, where $\Z[\Z_{n_i}]$ denotes the integral group
ring of $\Z_{n_i}$. The multiplication is given by: for
$\alpha \in \Z[\Z_{n_i}]$ and $\beta \in \Z[\Z_{n_j}]$,
\[
  \alpha \cdot \beta =
  \begin{cases}
    \alpha\beta & \text{if } i = j, \\
    0           & \text{if } i \neq j,
  \end{cases}
\]
where $\alpha\beta$ denotes multiplication in $\Z[\Z_{n_i}]$.
\end{proposition}
\begin{proof}
Since $\calG$ is totally disconnected, we have
$\Mor(\calG) = \bigsqcup_{i=1}^{m} \Hom(i,i) = \bigsqcup_{i=1}^{m} \Z_{n_i}$.
The free abelian group on this set decomposes as
$\bigoplus_i \Z[\Z_{n_i}]$.

For morphisms $g \in \Hom(i,i)$ and $h \in \Hom(j,j)$, we have
$\operatorname{sr}(g) = \operatorname{tr}(g) = i$ and $\operatorname{sr}(h) = \operatorname{tr}(h) = j$. The condition $\operatorname{sr}(g) = \operatorname{tr}(h)$
holds if and only if $i = j$. When $i = j$, the product
$g \cdot h = g \circ h$ is the group operation in $\Z_{n_i}$, extended
linearly to give multiplication in $\Z[\Z_{n_i}]$.
\end{proof}

\begin{remark}
From Proposition~\ref{prop-tdg1}, it follows that for finite $m$, we have an isomorphism of rings
\[
  \Z[\calG] \;\cong\; \prod_{i=1}^{m} \Z[\Z_{n_i}],
\]
where multiplication is component-wise. 
\end{remark}

As remarked earlier,  $\Z[\calG]$ is a ring, not a category. Nonetheless, notice that the set $\{ e_i \}$ of local units implicitly encodes the objects of $\calG$ within the ring. In fact, one can go further and express the data of this groupoid ring in manifestly categorical terms, which will further illuminate its structure.  
To achieve this, let us first define a $\Z$-Algebroid, following \cite{mitchell1985, Kelly1982, mitchener2014, avi2026, nlab:ab-enriched_category}:
\begin{definition}  
A \emph{$\Z$-algebroid} (also called an \emph{Ab-enriched category} or
a \emph{ringoid}) is a category $\mathcal{A}$ such that:
\begin{enumerate}  
  \item Each hom-set $\Hom_\mathcal{A}(x, y)$ carries the structure of an
        abelian group.
  \item Composition
        $\Hom_\mathcal{A}(y,z) \times \Hom_\mathcal{A}(x,y) \to \Hom_\mathcal{A}(x,z)$
        is $\Z$-bilinear.  
\end{enumerate}
Furthermore, the composition map satisfies the associativity and unit axioms of a category. 
\end{definition}

Next, we have \cite{Kelly1982, mitchell1985, mitchener2014}:    
\begin{definition}      
A \emph{functor of $\Z$-algebroids} (or \emph{Ab-enriched functor})    
$F : \mathcal{A} \to \mathcal{B}$ is a functor such that each map
$\Hom_F (x,y) : \Hom_\mathcal{A}(x,y) \to \Hom_\mathcal{B}(Fx, Fy)$ is a group  homomorphism, and is compatible with compositions and units in $\mathcal{A}$ and $\mathcal{B}$.   
\end{definition}   

The following definition then relates back to the integral groupoid ring that we have discussed above \cite{mitchell1985, mitchener2014, avi2026, nlab:ab-enriched_category}:    
\begin{definition}  
Given a groupoid $\calG$, its \emph{free $\Z$-linearization} is the
$\Z$-algebroid $\Zb[\calG]$ defined by:
\begin{enumerate}  
  \item $\Ob(\Zb[\calG]) = \Ob(\calG)$.
  \item $\Hom_{\Zb[\calG]}(x,y) = \Z[\Hom_{\calG}(x,y)]$, the free abelian
        group on $\Hom_{\calG}(x,y)$.
  \item Composition given by $\Z$-bilinear extension of groupoid
        composition.
\end{enumerate}
\end{definition}

The groupoid $\calG$ and its linearization $\Zb[\calG]$ share the same objects and underlying morphisms. However, they are distinct categorical structures: $\calG$ is an ordinary category (in fact, a groupoid as a category) with 
hom-sets, while $\Zb[\calG]$ is a $\Z$-algebroid with hom-groups. 
Nonetheless, there is a way to translate between the data that the two structures carry via a canonical inclusion functor 
$$\iota : \calG \to \Zb[\calG]$$  
sending each morphism $g$ to the basis element $1 \cdot g$. This enables the following isomorphism, which will be useful later, when we discuss the module theory:    
\begin{proposition}
\label{prop4.1v1}  
For any $\Z$-algebroid $\mathcal{A}$, the restriction along $\iota$ induces
a natural isomorphism of categories
\[
  \AbFun(\Zb[\calG],\, \mathcal{A}) \;\cong\; \Fun(\calG,\, \mathcal{A}),
\]
where $\AbFun$ denotes the category of Ab-enriched functors. 
\end{proposition}
\begin{proof}
Given an ordinary functor $F : \calG \to \mathcal{A}$, define
$\tilde{F} : \Zb[\calG] \to \mathcal{A}$ on objects by
$\tilde{F}(x) = F(x)$ and on morphisms by $\Z$-linear extension:
$\tilde{F}\!\left(\sum_i a_i g_i\right) = \sum_i a_i F(g_i)$. This is
well-defined and Ab-enriched. Conversely, any Ab-enriched functor
$\Zb[\calG] \to \mathcal{A}$ restricts to an ordinary functor
$\calG \to \mathcal{A}$. These constructions are mutually inverse.
\end{proof}

\subsection{Modules over Groupoid Rings \& $\Z$-Algebroids}

We now discuss the module theory for $\Z[\calG]$, as well as $\Zb[\calG]$.  

Let us begin with $\Z[\calG]$. A left $\Z[\calG]$-module extends the usual definition of a left ring module, taking into account the existence of local units  (\cite{Wisbauer1991, AndersonFuller1992, nastasescu2004methods}): 
\begin{definition}  
\label{def:alg-module}
A \emph{left $\Z[\calG]$-module} is an abelian group $M$ together with
a $\Z$-bilinear action $\Z[\calG] \times M \to M$, written
$(\alpha, \m) \mapsto \alpha \cdot \m$, satisfying:
\begin{enumerate}  
  \item  The action 
        $(\alpha\beta) \cdot \m = \alpha \cdot (\beta \cdot \m)$ is associative 
        for all $\alpha, \beta \in \Z[\calG]$, $\m \in M$.
  \item Local unitality decomposes $M$ as a direct sum
        \[
          M = \bigoplus_{x \in \Ob(\calG)} M_x,
          \quad \text{where } M_x := e_x \cdot M
          = \{\m \in M : e_x \cdot \m = \m\} 
        \]  
        where the set $\{ e_x \}$ are idempotent and orthogonal local units of $\Z[\calG]$ (from Definition \ref{ludef}).  
  \item Each morphism $g : x \to y$ in
        $\calG$ defines a graded left action, which realizes a group
        homomorphism $g_* : M_x \to M_y$.
  \item  Since $\calG$ is a groupoid, each
        $g_*$ becomes an isomorphism with an inverse  
        $(g^{-1})_* : M_y \to M_x$.
\end{enumerate}
Furthermore, a morphism  $f : M \to N$ of left $\Z[\calG]$-modules is a group
homomorphism satisfying $f(\alpha \cdot \m) = \alpha \cdot f(\m)$ for all
$\alpha \in \Z[\calG]$, $\m \in M$.
\end{definition}

The structure of $\Z[\calG]$ leads to a canonical decomposition of modules via idempotents.  To see how this works, let us first express the idempotents of $\Z[\calG]$ explicitly.   
As per Definition~\ref{ludef},  for each $1 \le i \le m$,  the \emph{$i$-th idempotent}   $e_i \in \Z[\calG]$ is the identity morphism
$\mathrm{id}_i \in \Hom(i,i) = \Z_{n_i}$, that is, the element
$0 \in \Z_{n_i}$ viewed as a morphism. Writing $\Z[\calG]$ as  
$\Z[\calG] \cong \prod_{i=1}^{m} \Z[\Z_{n_i}]$, the idempotent elements can be expressed in the following notation 
\[
  e_i = (0, \ldots, 0,\, \mathbf{1}_i,\, 0, \ldots, 0),
\]
where $\mathbf{1}_i \in \Z[\Z_{n_i}]$ is the identity element
(corresponding to $0 \in \Z_{n_i}$) in position $i$.

\begin{remark}
Notice that the idempotents of $\Z[\calG]$, described above, also satisfy completeness:
  $$\sum_{i=1}^{m} e_i = 1_{\Z[\calG]}$$  
which directly follows from the from component-wise multiplication in
$\prod_i \Z[\Z_{n_i}]$.
\end{remark}

Using the above, we can now show the module decomposition of $\Z[\calG]$, given $\calG = \calG(n_1, \ldots, n_m)$, as follows (\cite{nastasescu2004methods}): 
  
\begin{proposition}    
\label{prop-id}  
For  $\calG = \calG(n_1, \ldots, n_m)$ (Definition~\ref{defgcygr}), and  $M$  a $\Z[\calG]$-module, with   $M_i := e_i \cdot M$ for each  $i$, we have that:  
\begin{enumerate} 
  \item Each $M_i$ realizes a $\Z[\Z_{n_i}]$-module.
  \item The idempotents $\{ e_i \}$ uniquely specify the module decomposition  
        \[
          M \cong \bigoplus_{i=1}^{m} M_i.
        \]         
\end{enumerate}
\end{proposition}
\begin{proof}
(1) The subset $e_i M = \{e_i \cdot \m: \m\in M\}$ is an abelian
subgroup of $M$. For $\alpha \in \Z[\Z_{n_i}]$ and $e_i \m\in M_i$,
define
\[
  \alpha \cdot (e_i m) := \tilde{\alpha} \cdot \m 
\]
where $\tilde{\alpha} = (0, \ldots, \alpha, \ldots, 0) \in
\prod_j \Z[\Z_{n_j}] \cong \Z[\calG]$ with $\alpha$ in position $i$.
Since $e_i$ acts as the identity on $M_i$ and annihilates $M_j$ for
$j \neq i$, this is well-defined and gives a $\Z[\Z_{n_i}]$-module
structure.

(2) By completeness, every $\m \in M$ can be written as
\[
  \m = 1 \cdot \m= \left(\sum_{i=1}^{m} e_i\right) \cdot \m
  = \sum_{i=1}^{m} e_i \m
\]
with $e_i \m \in M_i$. Furthermore, by orthogonality, if $\sum_i \m_i = 0$ with
$\m_i \in M_i$, then applying $e_j$ gives $\m_j = 0$. Thus, $\sum_i \m_i$ is linearly independent for every $\m$, which uniquely specifies the   decomposition of $M$ as a $\Z[\calG]$-module.    
\end{proof}

Now let us discuss the module theory for $\Zb[\calG]$. This will follow the  general categorical definition of modules over rings with several objects, as     enriched functors  \cite{Mitchell1972, mitchell1985, Kelly1982,   mitchener2014}.  Here, we will restrict to the case $\Zb[\calG]$, where $\calG$ is a groupoid.   

\begin{definition}[ \cite{Mitchell1972, mitchell1985} ] 
\label{defZbmod}  
Let $\calG$ be a groupoid. A \emph{left $\Zb[\calG]$-module} is an
Ab-enriched functor to the category of abelian groups: 
\[
  M : \Zb[\calG] \longrightarrow \Ab  
\]
Such that, for each object $i \in \Ob(\calG)$, this assigns an abelian group $M_i$. 

And, for each morphism $a \in \Hom(i,j)$, this assigns a group homomorphism $M(a) : M_i \to M_j$. 

Furthermore, the above satisfy two functoriality axioms referring to the identity and composition of morphisms:  
\begin{enumerate}
    \item  $M(id_i) = \mathrm{id}_{M_i}$ for all objects $i$.
    \item  $M(b \circ a) = M(b) \circ M(a)$ for all composable morphisms $a, b$.
\end{enumerate}
\end{definition}

Moreover, we also have \cite{Kelly1982, mitchell1985}:   
\begin{definition} 
A \emph{morphism of left $\Zb[\calG]$-modules} is an Ab-enriched natural transformation
$\eta : M \Rightarrow N$, consisting of group homomorphisms
$\eta_x : M_x \to N_x$ for each $x \in \Ob(\calG)$ such that for
every $\alpha \in \Zb[\calG](x,y)$ the following compatibility condition holds: 
\[
  \eta_y \circ M(\alpha) = N(\alpha) \circ \eta_x.
\]
The module category is denoted by:  
\[
   \Zb[\calG]\textbf{-Mod}  := \AbFun(\Z[\calG],\, \Ab).
\]
\end{definition}

The following corollary will be useful for practical evaluations: 
\begin{corollary}  
By Proposition~\ref{prop4.1v1}, we have:  
\[
  \Zb[\calG]\textbf{-Mod}  \;\cong\; \Fun(\calG,\, \Ab).
\]
Therefore, specifying a $\Zb[\calG]$-module is equivalent to specifying an
ordinary functor $\calG \to \Ab$.
\end{corollary}

We can now apply the above definition to the totally disconnected groupoid  $\calG = \calG(n_1, \ldots, n_m)$.   Then, we have:   
\begin{remark} 
\label{remZbmod}
Given $\calG(n_1, \ldots, n_m)$,  a left $\Zb[\calG]$-module $M$  consists of the following:    
\begin{enumerate}  
  \item A collection of abelian groups $M_1, \ldots, M_m$.
  \item For each morphism $a \in \Hom(i,i) = \Z_{n_i}$,  a group homomorphism
$M(a) : M_i \to M_i$, which defines an action of $\Z_{n_i}$ on $M_i$.  
\end{enumerate}

And, functoriality axioms implying:       
\begin{enumerate}
  \item $M(0) = \mathrm{id}_{M_i}$ 
  \item $M(a+b) = M(a) \circ M(b)$ 
\end{enumerate}
 
Moreover, since there are no morphisms between distinct objects in $\calG$, the  modules $M_i$ are completely independent of one another.
\end{remark}

Placing everything together, it is now easy to see that the left modules of $\Z[\calG]$ and $\Zb[\calG]$ are equivalent. More generally, the equivalence between modules of a small $k$-algebroid, to the modules of a $k$-ring associated to it (where the ring forgets about the object set), has been known since \cite{Mitchell1972, mitchell1985}\footnote{In \cite{Mitchell1972}, these were referred to as matrix rings of a (pre)additive category.}. For the special case we are interested in here, the equivalence of modules can be seen using a straightforward construction.  

\begin{proposition}
There is an isomorphism of categories: 
\[  \Z[\calG]\textbf{-Mod}  \;\cong\;  \Zb[\calG]\textbf{-Mod}     \]
\end{proposition}
\begin{proof}
A left $\Z[\calG]$-module $M$ decomposes as $\bigoplus_{x} M_x$ with idempotents $e_x$ acting as projections onto $M_x$.  
Define a functor  $\tilde{M} : \Zb[\calG] \to \Ab$   by 
$$\tilde{M}_x := M_x  \qquad  \mbox{and}  \qquad  \tilde{M}(\alpha) := \alpha_* : M_x \to M_y$$ 
for $\alpha \in \Zb[\calG](x,y)$.  
The conditions in Definition~\ref{def:alg-module} ensure this is Ab-enriched.

Conversely, given an Ab-enriched functor $M : \Zb[\calG] \to \Ab$, form the abelian  group $\overline{M} = \bigoplus_{x \in \Ob(\calG)} M_x$. Then, define the action of $\alpha \in \Zb[\calG](x,y)$ on $\m \in M_x$ by
$$\alpha \cdot \m := (M(\alpha))(\m) \in M_y$$  
(and zero on other components). Then, $M(id_x)$ corresponds to the idempotent $e_x$, and the functoriality axioms of  Definition~\ref{defZbmod} ensure that Definition~\ref{def:alg-module} is satisfied.  
\end{proof}
Thus, $\Z[\calG]$ and $\Zb[\calG]$ can be thought of as different presentations of the same module theory, though the algebroid is the more natural categorical object.

\subsection{Multivariate Alexander Modules \& their Category of Elements}

Equipped with the formalism above, we now define multivariate Alexander modules from $\Z$-algebroids (and groupoid rings). First, let us recall that classical univariate Alexander modules are defined over the group ring  $\Z[\Z]$ of the infinite cyclic group. The underlying set consists of a  $\Lam$-module $M$, where $\Lam = \Z[t, t^{-1}]$  (isomorphic to $\Z[\Z]$) is the ring of Laurent polynomials in $t$, over $\Z$, with $t$ acting as a generator of deck transformations. We generalize this definition, working with modules over  groupoid rings and algebroids in the following way:    
 
\begin{definition}[Multivariate Alexander Module] 
\label{mamdef}
Let $\calG = \calG(n_1, \ldots, n_m)$ be the totally disconnected groupoid with  hom-sets $\Hom (i,i)$ as cyclic groups $\Z_{n_i}$.  
A \emph{multivariable Alexander (left) module} is an $\Ab$-enriched functor from the $\Z$-algebroid  $\Zb[\calG]$ to the category of abelian groups
\[
  M \;:\; \Zb[\calG] \longrightarrow \Ab
\]
Equivalently, it is a locally unital left module over the groupoid ring
$\Z[\calG]$, graded by idempotents $\{ e_i \}$  as 
\[  M = \bigoplus_{i = 1}^m  M_i   \]   
where each   $M_i = e_i M$  is a $\Z[\Z_{n_i}]$-module, with a choice of generator $t_i \in \Z_{n_i}$ for each $i$,  determining an isomorphism
\[ \Z[\Z_{n_i}] \cong \Z_{n_i}[t_i, t_i^{-1}]/( p_i (t_i) )  \] 
where the polynomial $p_i (t_i)$ is such that $t_i$ is invertible in the quotient.  
For an element $x \in M_i$, we write $t_x := t_i$ for its \emph{t-modulus}.
\end{definition}

Compared to the univariate Alexander module, wherein  the action of the deck group induces the module structure over the Laurent polynomial ring; in the multivariate case, we have that each vertex group $\Z_{n_i}$ of $\calG$ induces its own $\Z[\Z_{n_i}]$-module, with a designated generator $t_i$.  Thus, multivariate Alexander modules are oidified versions of Alexander modules, constructed from groupoids encoding $m$ deck groups.  

However, notice that the module $M$ in Definition~\ref{mamdef} is a functor (equivalently, a module over a groupoid ring), rather than a classical module. Its data includes the collection $\{ M_1, ... , M_m \}$, which are themselves  modules corresponding to individual objects in $\calG$. On the other hand, a quandle operation is defined over a set, associated to a module.  
In the univariate case, the underlying set of $M$ is simply the elements of $M$. Now how does one derive a set from a functor $M \colon \calG \to \Ab$ in $\Zb[\calG]\textbf{-Mod}$ $\!$?  

This can be done via a standard functorial construction, which involves ``unpacking" the structure of  $M$ to get the corresponding  \emph{category of elements}   \cite{grothendieck71, riehl2016category, nlab:category_of_elements}. The latter is a special case of the Grothendieck construction for covariant functors  \cite{grothendieck71}.  We then have (\cite{riehl2016category, nlab:category_of_elements}):    
\begin{definition}  
\label{defele}
Let $U \colon \Ab \to \Set$ be the forgetful functor from $\Ab$ to the category of sets $\Set$,  and
$M \colon \calG \to \Ab$ a module. Then,  
\[
  X(M) \;:=\; \Ob\!\left(\Gint{\calG} U \circ M\right)  
\]
where $\Gint{\calG}{U \circ M}$ is the \emph{category of elements} of
$U \circ M \colon \calG \to \Set$: its objects are pairs $(i, \m_i)$
with $\m_i \in M(i)$, and a morphism $(i, \m_i) \to (j, \m_j)$ is a morphism
$g \colon i \to j$ in $\calG$ with $(M(g))(\m_i) = \m_j$.  
\end{definition}

Furthermore, the projection functor  
\[  \pi \colon \Gint{\calG}{U \circ M}  \to \calG,   \qquad     (i, \m_i) \mapsto i  \] 
  is a discrete opfibration  \cite{riehl2016category, nlab:category_of_elements}.   Its fibre over $i$ is the discrete category on $M_i$.   In the sense of  $\Set$ as the classifying space of $\Set$-bundles, $\Gint{\calG}{U \circ M}$ is the $\Set$-bundle classified by  $U \circ M$  \cite{riehl2016category, nlab:category_of_elements}.

Then, by Definition~\ref{defele}, we get:    
\[
  \Ob\!\left(\Gint{\calG} U \circ M\right)
  \;=\; \bigsqcup_{i \in \Ob(\calG)} U(M(i))
  \;=\; \bigsqcup_{i \in \Ob(\calG)} M_i,
\]
Hence, the category of elements formally unveils the structure of the functor  $M$ into its tagged elements, with morphisms inherited from
$\calG$. In this way, $X(M)$ turns out to be precisely the (set-valued) disjoint union of modules $M_i$.  Moreover, since $\Hom_\calG(i,j) = \varnothing$ for $i \neq j$, all morphisms in $\Gint{\calG}{U \circ M}$  are endomorphisms, and therefore we have that $\Gint{\calG}{U \circ M}$  is itself a totally disconnected groupoid.  This circles back to our discussion in Section~3, with $X(M)$ thus reflecting the structure of the original groupoid $\calG$ encoding deck groups.

As an alternate construction, let us show how we can also extract the same set $X(M)$  from a module $M = \bigoplus_{i=1}^{m} M_i$  in $\Z[\calG]\textbf{-Mod}$;  even though this construction will be less aesthetic (from a manifestly categorical point of view) than the functorial one above.      

In this case, we will think of $M = \bigoplus_{i \in X} M_i$ as an $X$-graded module (see \cite{nastasescu2004methods}) with $M_i := e_i M$, and $X := \Ob(\calG) = \{ 1, ... ,m \}$.  Consider the product space $X \times M$, where $X$ is viewed as a discrete space. The projection $X \times M \to X$ can be thought of as the trivial bundle over $X$ with constant fiber $M$. We then have:
\begin{definition}  
The \emph{tagged set} underlying $M$ is 
\[
  X(M) := \{(i, \m) \in X \times M \mid e_i \m = \m\}    \]
with the projection map  
\[   \pi : X(M) \to X, \qquad (i, \m) \mapsto i
\]
\end{definition}

Notice that we have $e_i \m = \m \iff \m \in M_i$. Therefore, the fiber over $i$ is $\{i\} \times M_i$, which yields the disjoint union   
\[
  X(M) = \bigsqcup_i M_i   \quad \text{with} \quad  |X(M)| = \sum_i |M_i|  
\]
And the $\Z[\calG]$-action on $M$ then restricts to a fiber-wise action on $X(M)$. For  $\alpha \in  \Z[\Z_{n_i}]$ and $(j, \m_j) \in X(M)$, we have: 
\[
  \alpha \cdot (j, \m_j) =
  \begin{cases}
    (i,  \alpha \cdot  \m_i) & \text{if } j = i  \\
    0         &  \text{if } j \neq i
  \end{cases}
\]
with $e_i (\alpha \m_i) = e_i (e_i \alpha e_i \m_i)  = \alpha \m_i$ and  $\alpha \m_j = e_i \alpha e_i (e_j \m_j)  = 0$.  Moreover, since $\alpha \in \Hom(i,i) = \Z_{n_i}$,  the action reduces to each vertex group acting on its own fiber.  
This makes each fiber $M_i$ a $\Z_{n_i}$-set (a $\Z[\Z_{n_i}]$-module), with no cross-fiber action. Thus, $X(M)$ is now realized as a graded set over $X = \Ob(\calG)$ with a fiber-wise $\Z[\calG]$-action (or, $\Z$-linearized $\calG$-action).

The above constructions, lead to the equivalence: 
\begin{remark}
 The tagged set of a $\Z[\calG]$-module $M$ is thus equivalent to the object set of the category of elements of the functor $U \circ M$, built from (the functor) $M$:    
\[
  X(M)
  = \Ob\!\left(\Gint{\calG} U \circ M\right)
  = \{(i, \m) : e_i \m = \m\}
  = \bigsqcup_i e_i M  
\]
with fiber-wise $\calG$-action given by the morphisms of  $\Gint{\calG} U \circ M$ (equivalently, by the elements of the ring $\Z[\calG]$).   
\end{remark}

\section{Multivariate Alexander Quandles}

Having a precise definition for the oidified Alexander modules and their category of elements, we now define multivariate quandles:    

\begin{definition}[Multivariate Alexander Quandle]  
\label{def:maquand}
Given $M$ a multivariate Alexander module, and 
\[
  X(M) \;:=\; \Ob\!\left(\Gint{\calG} U \circ M \right) \;=\; \bigsqcup_{i=1}^{|\Ob(\calG)|} M_i  
\]
its tagged set, i.e., the object set of the category
of elements of the underlying set-valued functor $U \circ M$,  a  \emph{multivariate Alexander quandle} on $M$ is a quandle $(X(M), *, \bar{*})$ satisfying the axioms (Definition~\ref{qaxioms}) of idempotence, right-invertibility, and right-self-distributivity, together with the following conditions: 

\begin{enumerate}  
  \item  The constraints:  
  $$t_{(a * b)} = t_{a}  \qquad \mbox{and}  \qquad  t_{(a \bar{*} b)} = t_a$$ 
  ensuring that the operation is closed on each fiber $M_i$ ($1 \leq i \leq |\Ob(\calG)|$), whenever $a \in M_i$, for every $b \in M$, such that $a * b \in M_i$  (likewise, for the inverse operation).   

  \item  A general affine operation:    
   \[
          (a_1, a_2) * (b_1, b_2) \;=\; \bigl(a_1, \; T_{ a_1, b_1 }\, a_2 + (1 - F_{ a_1, b_1 })\, b_2 \bigr)  
        \]
  where the tagged elements $a$ and $b$ are explicitly denoted as $(a_1, a_2)$ and  $(b_1, b_2)$ respectively, with tagging labels $a_1, b_1$; and the matrix elements $T_{ a_1, b_1 }$, $F_{ a_1, b_1 }$ are subject to the 3 quandle axioms and invertibility within the module.  
  
With inverse operation:  
  \[
          (a_1, a_2) \; \bar{*} \; (b_1, b_2) \;=\; \bigl(a_1, \; (T_{ a_1, b_1 })^{-1}\; a_2 + (T_{ a_1, b_1 })^{-1} \, (F_{ a_1, b_1 } - 1)\; b_2 \bigr)   
        \]

\item   When $a_1 = b_1$, that is, both $a$ and $b$ belong to the same fiber  $M_i$, associated to the invertible element $t_i$ (equivalently, $t_a$), the quandle operation reduces to the Alexander operation: 
  \[
          (a_1, a_2) * (a_1, b_2) \;=\; \bigl(a_1, \; t_{i}\, a_2 + (1 - t_{i})\, b_2 \bigr) 
        \]
\end{enumerate}
\end{definition}

Let us contrast Definition~\ref{def:maquand} to that of Traldi in \cite{traldi1}:
\begin{remark}   
In Traldi's construction \cite{traldi1}, the Alexander module $M_A(L)$ of a $m$-component link is a single module over the commutative Laurent polynomial ring $\Lam_m = \mathbb{Z}[t_1^{\pm 1}, \ldots, t_m^{\pm 1}]$, presented as the cokernel of a map associated to a link diagram. The total multivariate Alexander quandle is then a subset of the module $M_A(L)$, defined by the  epimorphism appearing in the link module sequence  (see   \cite{crowell1969elementary, hillman1981alexander}),  and all $m$ variables act on all module elements simultaneously \cite{traldi1}. 

In our construction here, the quandle is defined on $X(M)$, which is the object set of the category of elements  $\Gint{\calG} U \circ M$, and not on a subset of the module $M$. Moreover, each variable $t_i$ acts only on its own 
fiber $M_i$, defined by the action of $\Z_{n_i}$ on $M_i$.

When $m = 1$, both definitions recover the standard Alexander quandle.  
\end{remark}

Next, we determine explicit solutions for the matrices $T$ and $F$ appearing in Definition~\ref{def:maquand}. This will provide us with distinct classes of multivariate quandle operations. First, we shall consider the case when    $M_1, ... , M_m$ are all of identical order (though their $t$-moduli may well be different); and then, the case involving non-identical orders.

\subsection{Operations with $\Z$-Modules of Identical Order}

Consider a module $M$ with each $M_i$ of order $n$. The tagged set can be expressed as 
\[  X(M) = \bigsqcup_{i=0}^{m - 1} ( i, \;  \Z_{n_i} [t_{i+1}, t_{i+1}^{-1}]/p_i (t_{i+1}) )   \]   
where it shall be convenient to have the tagging index $i$ run from $0$ to $m-1$. The order of each $\Z_{n_i} [t_{i+1}, t_{i+1}^{-1}]/p_i (t_{i+1})$ above is 
$n_i^{deg (p_i)}$, which is fixed to $n$. The cardinality of $X(M)$ is $m \times n$.  In the notation $(a_1, a_2)$ for elements $a \in X(M)$, the tagging index $a_1$ can be thought of as elements in $\Z_m$, as in $a_1 \in \Z_m$.

We then have:  
\begin{proposition}
Given $M$ a (finite) multivariate Alexander module with each $M_i$ of fixed order $n$, the matrices $T = (T_{ a_1, b_1 })$ and $F = (F_{ a_1, b_1 })$ are completely specified by Definition~\ref{def:maquand}.  In particular, idempotence implies $T_{ a_1, a_1 } = F_{ a_1, a_1 } = t_a$,  and  right-self-distributivity with $t_{(a * b)} = t_{a}$, determines three general solution classes of multivariate quandle operations:  
\begin{eqnarray*}
  &\text{(i) }& T = F, \qquad  \;\,\   T_{ a_1, b_1 } = t_b   \\  
  &\text{(ii) }& T = F, \qquad  \;\,\  T_{ a_1, b_1 } = \delta_{ a_1, b_1 } \, t_a + (1 - \delta_{ a_1, b_1 })   \\  
  &\text{(iii) }& T_{ a_1, b_1 } = t_b, \hspace{,5cm}  F_{ a_1, b_1 } = t_a
\end{eqnarray*}
\end{proposition}

\begin{proof}
Given $a, b \in X(M)$, let us express a multivariate Alexander quandle operation of finite order as a specification of the pair:      
$$a * b = (a_1 \bullet b_1, a_2 \bullet b_2)$$   
where 
\begin{eqnarray}
a_1 \bullet b_1 &:=& a_1 \pmod{m}  \label{pr1}  \\
a_2 \bullet b_2 &:=& T(a_1, b_1) a_2 + (1 - F(a_1, b_1)) b_2 \pmod{n}   \label{pr2} 
\end{eqnarray}  
Here $T (a_1, b_1)$ and $F (a_1, b_1)$ are elements of $m \times m$ matrices (the notation used here is simply a more readable form of $T_{a_1, \, b_1}$). We now show that the above matrix elements are determined by imposing the quandle axioms. Moreover, the condition in eq.~(\ref{pr1})  turns  out to be a constraint on the parameters $t_i$.  

Firstly, notice that idempotence implies that the diagonal entries of $T$ and $F$ are identical. Next, imposing right self-distributivity yields the following terms on the left-hand side (of the self-distributivity axiom): 
\begin{eqnarray}
T(a_1, c_1) \left( T(a_1, b_1) a_2 + (1 - F(a_1, b_1)) b_2  \right) + (1 - F(a_1, c_1)) c_2
 \label{pr3} 
\end{eqnarray}
and, the following expression on the right-hand side: 
\begin{eqnarray}
T(a_1, b_1) \left( T(a_1, c_1) a_2 + (1 - F(a_1, c_1)) c_2  \right) + (1 - F(a_1, b_1)) \left( T(b_1, c_1) b_2 + (1 - F(b_1, c_1)) c_2  \right) 
 \label{pr4} 
\end{eqnarray}
Here, eq.~(\ref{pr1}) has been used within the arguments of $T$ and $F$. Matching the coefficients of $a_2, b_2, c_2$ respectively in eqs.~(\ref{pr3}) and (\ref{pr4}), we find the following three solution classes for $T$ and $F$: 

$Class \; (i)$:  When the off-diagonal entries of $T$ and $F$ are generally not identical
\begin{eqnarray}
T(a_1, c_1) &=& T(b_1, c_1)   \nonumber \\
F(a_1, c_1) &=& F(b_1, c_1)   \nonumber \\
T(a_1, b_1) &=& F(a_1, b_1)   \nonumber 
\end{eqnarray}
for any $a_1, b_1, c_1$. This implies $T = F$ as matrices, and the columns of $T$ have identical $t$-moduli
\[ T =  \begin{pmatrix}
    t_{1} & t_{2} & \cdots & t_{m} \\
    t_{1} & t_{2} & \cdots & t_{m} \\
    \vdots & \vdots & \ddots & \vdots \\
    t_{1} & t_{2} & \cdots & t_{m}
\end{pmatrix}    \]
Here $t_i$ is the $t$-modulus associated to $\Z_{n_i} [t_{i+1}, t_{i+1}^{-1}]/p_i (t_{i+1})$ with $1 \leq i \leq m$.  

With this form of $T$ and $F$, the quandle multiplication can be expressed as
\begin{eqnarray}
    a_1 \bullet b_1 &=& a_1 \pmod{m}  \label{op1} \\
    a_2 \bullet b_2 &=& t_{b} a_2 + (1 - t_{b}) b_2 \pmod{n}   \label{op2} 
\end{eqnarray}
where $t_{b}$ denotes the $t$-modulus associated to the element $b = (b_1, b_2)$.

$Class \; (ii)$: The next solution case we have is when the off-diagonal entries of $T$ and $F$ are identical. Then, 
\begin{eqnarray}
T(a_1, b_1) &=& 1   \nonumber \\
F(a_1, b_1) &=& 1   \nonumber \\
T(a_1, a_1) &=& F(a_1, a_1)   \nonumber 
\end{eqnarray}
for any $a_1, b_1$, which also implies $T = F$ as matrices, and $T$ takes the form
\[ T =  \begin{pmatrix}
    t_{1} & 1 & \cdots & 1 \\
    1 & t_{2} & \cdots & 1 \\
    \vdots & \vdots & \ddots & \vdots \\
    1 & 1 & \cdots & t_{m}
\end{pmatrix}    \]

Then, the quandle multiplication takes the form 
\begin{eqnarray}
    a_1 \bullet b_1 &=& a_1 \pmod{m}    \\
    a_2 \bullet b_2 &=& \left( \delta_{ a_1, b_1 } \, t_a + (1 - \delta_{ a_1, b_1 }) \right) a_2 + (1 - \left( \delta_{ a_1, b_1 } \, t_a + (1 - \delta_{ a_1, b_1 }) \right) ) b_2 \pmod{n}     
\end{eqnarray}

$Class \; (iii)$:  Finally, the equality of eqs.~(\ref{pr3}) and (\ref{pr4}) is also satisfied when $T(a_1, b_1) = t_b$ and $F(a_1, b_1) = t_a$. That is, as functions of $a_1$ and $b_1$, the matrix elements $T(a_1, b_1)$ and $F(a_1, b_1)$ map to the $t$-moduli corresponding to their second and first arguments respectively. Here, $t_{b}$ is associated to the element $b = (b_1, b_2)$, and $t_{a}$ is associated to the element $a = (a_1, a_2)$.  This gives the following matrix forms for $T$ and $F$: 
\[ T =  \begin{pmatrix}
    t_{1} & t_{2} & \cdots & t_{m} \\
    t_{1} & t_{2} & \cdots & t_{m} \\
    \vdots & \vdots & \ddots & \vdots \\
    t_{1} & t_{2} & \cdots & t_{m}
\end{pmatrix}     \quad  \quad 
F =  \begin{pmatrix}
    t_{1} & t_{1} & \cdots & t_{1} \\
    t_{2} & t_{2} & \cdots & t_{2} \\
    \vdots & \vdots & \ddots & \vdots \\
    t_{m} & t_{m} & \cdots & t_{m}
\end{pmatrix}    
\]

These matrices then lead to the quandle multiplication mentioned in Section~2.  
\begin{eqnarray}
    a_1 \bullet b_1 &=& a_1 \pmod{m}    \label{op5}  \\
    a_2 \bullet b_2 &=& t_{b} a_2 + (1 - t_{a}) b_2 \pmod{n}   \label{op6} 
\end{eqnarray}

\end{proof}

\begin{remark} 
\label{remtab}
Notice that, eq.~(\ref{pr1}) is equivalent to the constraint $t_{a\,*\,b} = t_a$ (since the components $a_1$ and $b_1$ refer to the module label, eq.~(\ref{pr1}) picks the label referring to the module associated to  $a$ in the product $a\,*\,b$).     

Moreover, for each of the above three solution classes,  eq.~(\ref{pr1}) is   crucial for eqs.~(\ref{pr3}) and (\ref{pr4}) to hold. If one tries to generalize eq.~(\ref{pr1}) to an operation of the form $a_1 \bullet b_1 = t a_1 + (1 - t) b_1$ for some $t \neq 1$, then self-distributivity immediately breaks down. In other words, the above three solution classes of multivariate Alexander quandle   operations only hold when $t_{a\,*\,b} = t_a$ is satisfied.  
\end{remark}

\subsection{Operations with $\Z$-Modules of Non-Identical Order}

We now consider the more general set-up, where $\Z$-modules $M_i$ of generally non-identical order are taken.  The tagged set is 
\[  X(M) = \bigsqcup_{i=0}^{m - 1} M_{i}  \qquad \mbox{with cardinalities}  \qquad  | M_{i} | = n_{i}  \]
 In this case, eqs.~(\ref{pr1}) and (\ref{pr2}) are replaced by: 
\begin{eqnarray}
a_1 \bullet b_1 &=& a_1 \pmod{m}  \label{pr1v2}  \\
a_2 \bullet b_2 &=& T(a_1, b_1) a_2 + (1 - F(a_1, b_1)) b_2  \pmod{n_{i}}   \label{pr2v2} 
\end{eqnarray}  
where $\pmod{n_{i}}$  refers to the $\Z$-module containing the element $a$.  

Once again, we can impose quandle axioms on eqs.~(\ref{pr1v2}) and (\ref{pr2v2}) to determine solutions to $T$ and $F$. However, the generality of non-identical orders makes things more nuanced.  The following holds: 
 
\begin{proposition}  
Given $M$ a multivariate Alexander module with each $M_i$ not necessarily of  identical order,  then eqs.~(\ref{pr1v2}) and (\ref{pr2v2})  yields a valid multivariate Alexander quandle operation when $T = F$ and  
\[ T =  \begin{pmatrix}
    t_{1} & 1 & \cdots & 1 \\
    1 & t_{2} & \cdots & 1 \\
    \vdots & \vdots & \ddots & \vdots \\
    1 & 1 & \cdots & t_{m}
\end{pmatrix}    \]
\end{proposition}

Notice that these are exactly the $class \; (ii)$ solutions that we have seen in the previous subsection. Let us prove that these solutions always exist when modules of generally non-identical order are taken. 

\begin{proof}
The proof can be constructed iteratively. First, consider the case for  $m = 2$, that is, $X(M)$ consists of $(0, \Z_{r_0} [t_{1}, t_{1}^{-1}]/p_0 (t_{1}))$  and  $(1, \Z_{r_1} [t_{2}, t_{2}^{-1}]/p_1 (t_{2}))$, with cardinalities $n_0 = r_0^{deg (p_0)}$ and $n_1 = r_1^{deg (p_1)}$ respectively.   

Now assume $T = F$, given by the matrix form:
\begin{eqnarray}  
\begin{pmatrix} t_1 & 1 \\ 1 & t_2 \end{pmatrix}  
\label{teq10}
\end{eqnarray}
Substituting this in eqs.~(\ref{pr1v2}) and (\ref{pr2v2}) for $m = 2$, we will  check that they satisfies the quandle axioms. 

Note that the quandle multiplication table will now have rectangular off-diagonal blocks of dimensions $n_0 \times n_1$ and $n_1 \times n_0$ respectively. The structure of this multiplication table is of the form:
\begin{eqnarray}
\begin{array}{r@{\mskip\thickmuskip}c}
& \begin{array}{c@{\hspace{7.0em}}c} \scriptstyle 0 & \scriptstyle 1 \end{array} \\[-0.5ex]
\begin{array}{r} \scriptstyle 0 \\ \scriptstyle 1 \end{array} & 
 \left(\begin{array}{c|c}
dim (n_0 \times n_0) & dim (n_0 \times n_1) \\ \hline
dim (n_1 \times n_0) & dim (n_1 \times n_1)
\end{array}\right) 
\end{array}
\label{tab1}
\end{eqnarray}
where the block-labels ("0" or "1" above and to the left of the table) denote the tagged coordinates, which refer to the elements of the $T$-matrix (for instance, the element $T_{0\,1}$ is the $t$-variable associated to entries in the upper right block of the quandle table above).  To keep track of the specific $\Z$-module $M_i$  referring to each element $a \equiv (a_1, a_2)$, we will hereon denote the $a_2$ by $a_2 (a_1)$ (for instance, $a_2 (0)$ refers to the module with $a_1 = 0$). 

For the multiplication table in eq.~(\ref{tab1}), idempotence and right-invertibility follow immediately since the diagonal blocks themselves are quandle tables, and the  entries in the rows of the off-diagonal blocks are determined by the trivial quandle using $t = 1$ (which implies every column in the table in eq.~(\ref{tab1}) is a permutation).   

Now, to check self-distributivity, we will work with triples $\left( a_2 (i), \, b_2 (j), \, c_2 (k) \right)$, where $i, j, k \in \{0, 1\}$.  We need to check the following $2^3$ combinations of these triples, which we can compute explicitly using the $T$-matrix given in eq.~(\ref{teq10}) above:       

$(i) \; \& \; (ii)$  $(a_2(0), \, b_2(0), \, c_2(0))$ and $(a_2(1), \, b_2(1), \, c_2(1))$ are both trivially self-distributive since they belong to the same module.
    
$(iii)$    $(a_2(0), \, b_2(0), \, c_2(1))$ evaluates to $( a_2(0) \bullet b_2(0) )$ on both sides of the self-distributivity identity.  
   
$(iv)$   $(a_2(0), \, b_2(1), \, c_2(0))$ gives $( a_2(0) \bullet c_2(0) )$ on both sides.
   
$(v)$    $(a_2(0), \, b_2(1), \, c_2(1))$ leads to $a_2(0)$ on both sides of the identity.

$(vi)$   $(a_2(1), \, b_2(0), \, c_2(1))$ gives $( a_2(1) \bullet c_2(1) )$.

$(vii)$   $(a_2(1), \, b_2(1), \, c_2(0))$ gives $( a_2(1) \bullet b_2(1) )$.

$(viii)$  $(a_2(1), \, b_2(0), \, c_2(0))$ gives $a_2(1)$.

Since all 8 of the above combination of triples satisfy right self-distributivity, the multiplication table in eq.~(\ref{tab1}) as a whole satisfies right self-distributivity for any disjoint union of $M_0$ and $M_1$ with cardinalities $n_0$ and $n_1$ respectively. This yields a quandle for $m = 2$.   

One can now continue the above process iteratively: we append another $\Z$-module $M_2$ to the above quandle, and express the resulting table once again in $2 \times 2$ block form. The quandle with $X(M) = M_0 \, \bigsqcup \,  M_1$ is now treated as a single diagonal block.  As before, the off-diagonal blocks are determined by the trivial quandle with $t = 1$; and analogous to the case above, we have 8 combinations of triples of elements, whose evaluations confirm self-distributivity in the same way. Such an iterative application can be carried out for any finite number of $\Z$-modules $M_i$. This concludes our proof.   
\end{proof}

\begin{remark}
In the case of modules $M_i$ of non-identical order, $class \; (i)$ solutions do not exist, by construction. Also, $class \; (iii)$ solutions, only exist in certain cases (we show one such explicit example in the subsection on enumeration).
\end{remark}

What kind of algebras are the quandles we have constructed here? To answer this, we have the following: 
\begin{proposition}
Multivariate Alexander quandles are semisimple algebras.  
\label{prop5.2v0}
\end{proposition}

\begin{proof}  
Multivariate Alexander quandles necessarily satisfy $t_{(a * b)} = t_a$, which is a necessary condition for the construction of multivariate quandles when the abelianization map is generalized to distinguish between link components. Equivalently, this condition also follows from requiring right-self-distributivity of the multivariate quandle operation in eq.~(\ref{pr2}) (see Remark \ref{remtab}).  

Now, notice that $t_{(a * b)} = t_a$ realizes a right ideal on the quandle algebra such that the right-action via quandle multiplication preserves the $t$-variable of the associated ideal. Recall that each element in the module comes associated to a specific $t$. Since right action preserves $t$-variables, it realizes an ideal of the algebra defined by the corresponding $t$. The entire algebra is a disjoint union of elements labelled by particular $t$'s. Hence, the algebra decomposes into a disjoint union of proper ideals (including possible $1$-dimensional ideals).
\end{proof}

\section{Quiver Presentation }

Given how multivariate Alexander modules consisting of a collection of $\Z$-modules yield well-defined quandle algebras, one may ask whether this process of composing quandles of lower orders to generate larger quandle algebras suggests an underlying quiver structure for multivariate quandles.   

First, let us note the definition of a quiver (originating in Gabriel's seminal paper  \cite{gabriel1972unzerlegbare}; see also \cite{crawley1992lectures}):
\begin{definition}
A \emph{Quiver} is a finite directed graph written as a quadruple: ${\cal Q} = \{ V, E, \operatorname{sr}, \operatorname{tr} \}$, where $V$ is a finite set of vertices, $E$ is a finite set of edges (arrows), and $\operatorname{sr}, \operatorname{tr} : E \to V$ are maps assigning to each edge its source and target nodes respectively.  
\end{definition}

Now let us see how a multivariate quandle could inherit a quiver presentation. For the vertices, we take the  $\Z$-modules $M_i$ within the disjoint union. The edges of the quiver are specified by quandle actions, which we define as follows:  
\begin{definition}  
A quandle action on a $\Z$-module $M_i$ is a map $Q_{i j} : M_i \to M_j$ for some $j$, specified by 
\[  * : M_j \times M_i \to M_j    \]
where * denotes the standard quandle operation.
\end{definition}  
In fact, the entries in the off-diagonal blocks (those indexed $j\,i$) of the multivariate Alexander quandles above, are precisely generated by the action  $Q_{i j}$.  Note however, that in general, the matrix elements $Q_{i j}$ do not always generate a quandle algebra, except when the modules $M_i$ and $M_j$ are of identical order. Following our construction of multivariate Alexander quandles, let us make two observations:   
\begin{remark} 
(i) Associated to every Alexander quandle (of arbitrary number of variables), there is an underlying quiver presentation ${\cal Q}$, whose vertices are $\Z$-modules $M_i$, and directed edges are quandle actions $Q_{i j}$. 

(ii) Conceptually, the quiver presentation of a quandle, makes explicit, the oidification of Alexander modules.   
\end{remark} 
Note that the quiver presentation of a quandle is not one-to-one, but merely  serves as a useful tool for elucidating composability of quandles from those of lower orders (likewise, the reducibility of quandles into those of lower orders). 

\begin{example}
As an example, the following diagram (Fig.~(\ref{qfig0})) depicts the quiver presentation of  multivariate Alexander quandles of order 6. All of these 10 quandles involve two $\Z$-modules, hence the structure of the quiver diagram is identical. For multivariate Alexander quandles of more general orders, the associated quiver diagrams will always be complete directed  graphs. 

\begin{figure}[htbp]
\begin{center}
\begin{tikzcd}[
    cells={nodes={draw, circle, minimum size=2.5em}},
    column sep=7em
]
M_1 
  \arrow[r, "Q_{12}", bend left=45]
  \arrow[out=210, in=150, loop, distance=3.5em, "Q_{11}"{above, xshift=-0.7em, yshift=-0.5em}] 
& 
M_2 
  \arrow[l, "Q_{21}", bend left=45]
  \arrow[out=330, in=30, loop, distance=3.5em, "Q_{22}"{above, xshift=0.8em, yshift=-0.5em}]
\end{tikzcd}   
\caption{Quiver presentation of multivariate Alexander quandles with two  $\Z$-modules $M_1$ and $M_2$. }  
\label{qfig0}   
\end{center}  
\end{figure}
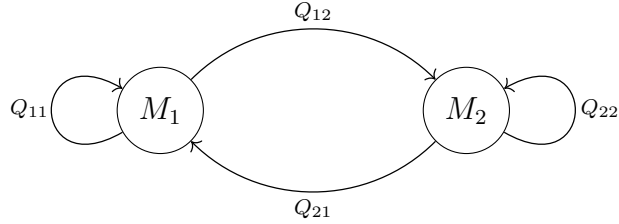

\end{example}

\section{Computations }  

\subsection{Enumeration }

Putting the all above constructions to the test, let us now compute a few concrete examples. The groupoid-modules constructed above lead to new isomorphism classes of Alexander quandles, those with multiple $t$ variables. This extends previous enumerations of Alexander quandles discussed in    \cite{nelson1, nelson2, lopes2006finite,  vendramin2012, clark2014quandle, higashitani2024classification}.  

For instance, let us explicitly enumerate all possible Alexander quandles of order 6. In the case of a single $t$-variable, there are only 2 non-isomorphic Alexander quandles, respectively based upon modules:  
\[  \mathbb{Z}_6 [t, \, t^{-1}] / (t-1)  \qquad  \mbox{and}  \qquad   \mathbb{Z}_6 [t, \, t^{-1}] / (t-5)  \]
On the other hand, when multiple $t$-variables are included, we find new examples of non-isomorphic multivariate Alexander quandles of order 6, described below.

\begin{example}
\label{eg5.1.0}
 Firstly, we consider the disjoint union   
\[  \mathbb{Z}_3 [t, \, t^{-1}]  / (t-1) \; \bigsqcup  \; \mathbb{Z}_3 [t, \, t^{-1}]  / (t-2) \] 
The elements of this tagged set $X(M)$ are:  
$$ (0,0),  (0,1),  (0,2),  (1,0),  (1,1),  (1,2)  $$ 
with the assignment $t =1$ for all elements with $a_1 = 0$, and $t =2$ for all elements with $a_1 = 1$. 

For this case, we find three non-isomorphic quandles:  
\begin{enumerate}
    \item   A quandle of solution $class \, (i)$ corresponding to the operations in eqs.~(\ref{op1}) and (\ref{op2}), with $T$-matrix:    
    \[  T =  \begin{pmatrix}
    1 & 2  \\
    1 & 2
\end{pmatrix}     \]  
and Cayley table:   
$$  \left(  \begin{array}{cccccc}   
 1 & 1 & 1 & 1 & 3 & 2 \\
 2 & 2 & 2 & 3 & 2 & 1 \\
 3 & 3 & 3 & 2 & 1 & 3 \\
 4 & 4 & 4 & 4 & 6 & 5 \\
 5 & 5 & 5 & 6 & 5 & 4 \\
 6 & 6 & 6 & 5 & 4 & 6 
\end{array}  \right)    $$
where, for purposes of presenting the Cayley table, it is convenient to map the elements of $X(M)$ to integers from $1$ to $6$. It is of course straightforward to check that the table in this form satisfies all three quandle axioms. 

Furthermore, note that a symmetric swap of the values of the $t$-parameters in $T$, such that $1 \leftrightarrow 2$, does not produce a new quandle - it merely gives a $T$-matrix that yields an isomorphic quandle to this one. 

\item  A quandle of solution $class \, (ii)$ where quandle multiplication is given  by the $T$-matrix:   
    \[  T =  \begin{pmatrix}
    1 & 1  \\
    1 & 2
\end{pmatrix}     \]  
leading to the Cayley table:
$$  \left(  \begin{array}{cccccc}  
 1 & 1 & 1 & 1 & 1 & 1 \\
 2 & 2 & 2 & 2 & 2 & 2 \\
 3 & 3 & 3 & 3 & 3 & 3 \\
 4 & 4 & 4 & 4 & 6 & 5 \\
 5 & 5 & 5 & 6 & 5 & 4 \\
 6 & 6 & 6 & 5 & 4 & 6
\end{array}  \right)   $$

\item   A quandle of solution $class \, (iii)$, corresponding to the operations in eqs.~(\ref{op5}) and (\ref{op6}). The Cayley table for this is:  
$$  \left(  \begin{array}{cccccc}  
    1 & 1 & 1 & 1 & 1 & 1 \\
    2 & 2 & 2 & 3 & 3 & 3 \\
    3 & 3 & 3 & 2 & 2 & 2 \\
    4 & 6 & 5 & 4 & 6 & 5 \\
    5 & 4 & 6 & 6 & 5 & 4 \\
    6 & 5 & 4 & 5 & 4 & 6
\end{array}  \right)    $$

\end{enumerate}

\end{example}

\begin{example}
Next, we consider the disjoint union  
$$\mathbb{Z}_3 [t, \, t^{-1}] / (t-2) \; \bigsqcup  \; \mathbb{Z}_3 [t, \, t^{-1}]  / (t-2)$$  
which gives a quandle of $class \, (ii)$, whose Cayley table is obtained by fusing 2 copies of quandles over $\mathbb{Z}_3 [t, \, t^{-1}]  / (t-2)$ (along the two diagonal blocks of the Cayley table), and 2 copies of the trivial quandle over mod $3$ (along the two off-diagonal blocks) as follows:  
\[  \left(\begin{array}{c|c}   
\;\;   \mathbb{Z}_3 [t, \, t^{-1}]  / (t-2)  \;\;  & \;\;  \mathbb{Z}_3 [t, \, t^{-1}]  / (t-1) \;\;  \\  &  \\  \hline  \\  
\mathbb{Z}_3 [t, \, t^{-1}]  / (t-1) & \mathbb{Z}_3 [t, \, t^{-1}]  / (t-2)
\end{array}\right)       \]    
This yields   
$$
\left(
\begin{array}{cccccc}
 1 & 3 & 2 & 1 & 1 & 1 \\
 3 & 2 & 1 & 2 & 2 & 2 \\
 2 & 1 & 3 & 3 & 3 & 3 \\
 4 & 4 & 4 & 4 & 6 & 5 \\
 5 & 5 & 5 & 6 & 5 & 4 \\
 6 & 6 & 6 & 5 & 4 & 6 \\
\end{array}
\right)
$$

\end{example}

\begin{example}
Moving further, we also have disjoint unions consisting of non-identical orders.  The first of these is  
$$\mathbb{Z}_2  [t, \, t^{-1}]   / (t-1) \;  \bigsqcup  \; \mathbb{Z}_4  [t, \, t^{-1}]   / (t-3)$$   
which gives two non-isomorphic quandles: 
\begin{enumerate}
    \item   One of $class \, (ii)$, with Cayley table:   
   $$  \left(  \begin{array}{cccccc}  
 1 & 1 & 1 & 1 & 1 & 1 \\
 2 & 2 & 2 & 2 & 2 & 2 \\
 3 & 3 & 3 & 5 & 3 & 5 \\
 4 & 4 & 6 & 4 & 6 & 4 \\
 5 & 5 & 5 & 3 & 5 & 3 \\
 6 & 6 & 4 & 6 & 4 & 6
\end{array}  \right)    $$

\item And, the other referring to $class \, (iii)$, with Cayley table:   
   $$  \left(  \begin{array}{cccccc} 
 1 & 1 & 1 & 1 & 1 & 1 \\
 2 & 2 & 2 & 2 & 2 & 2 \\
 3 & 5 & 3 & 5 & 3 & 5 \\
 4 & 6 & 6 & 4 & 6 & 4 \\
 5 & 3 & 5 & 3 & 5 & 3 \\
 6 & 4 & 4 & 6 & 4 & 6  \\
\end{array}  \right)    $$

\end{enumerate}  
The off-diagonal blocks of the above two quandles are now rectangular, and hence not complete quandles by themselves. Nonetheless, these tables  satisfy the required axioms. 

\end{example}

\begin{example}
Then, we have 
$$\mathbb{Z}_2  [t, \, t^{-1}]   / (t-1) \; \bigsqcup  \; \mathbb{Z}_2  [t, \, t^{-1}]   / (t^2 + t + 1)$$  
which gives a quandle of $class \, (ii)$, where the diagonal blocks are the corresponding quandles of orders $2$ and $4$ respectively, and the off-diagonal blocks are identical to those seen above in the $class \, (ii)$  solution of  $\mathbb{Z}_2  [t, \, t^{-1}]   / (t-1) \; \bigsqcup  \; \mathbb{Z}_4  [t, \, t^{-1}]   / (t-3)$. The Cayley table is:     
$$  \left(
\begin{array}{cccccc}
 1 & 1 & 1 & 1 & 1 & 1 \\
 2 & 2 & 2 & 2 & 2 & 2 \\
 3 & 3 & 3 & 5 & 6 & 4 \\
 4 & 4 & 6 & 4 & 3 & 5 \\
 5 & 5 & 4 & 6 & 5 & 3 \\
 6 & 6 & 5 & 3 & 4 & 6 \\
\end{array}
\right)  $$
Note that there are no quandles of $class \, (iii)$  when one of the factors of the disjoint union originates from a module whose invertible element $t$ remains a variable. 

\end{example}

\begin{example}
Finally, we have 
$$\mathbb{Z}_1  [t, \, t^{-1}]   / (t-1) \; \bigsqcup  \; \mathbb{Z}_5  [t, \, t^{-1}]   / (t - k)$$  
which gives three non-isomorphic quandles of $class \, (ii)$, for $k = 2, 3, 4$ respectively. In this case, the $class \, (iii)$ quandles are identical to those of $class \, (ii)$. We get the following Cayley tables:  

$$  \left(
\begin{array}{cccccc}
 1 & 1 & 1 & 1 & 1 & 1 \\
 2 & 2 & 6 & 5 & 4 & 3 \\
 3 & 4 & 3 & 2 & 6 & 5 \\
 4 & 6 & 5 & 4 & 3 & 2 \\
 5 & 3 & 2 & 6 & 5 & 4 \\
 6 & 5 & 4 & 3 & 2 & 6 \\
\end{array}
\right)   $$

$$  \left(
\begin{array}{cccccc}
 1 & 1 & 1 & 1 & 1 & 1 \\
 2 & 2 & 5 & 3 & 6 & 4 \\
 3 & 5 & 3 & 6 & 4 & 2 \\
 4 & 3 & 6 & 4 & 2 & 5 \\
 5 & 6 & 4 & 2 & 5 & 3 \\
 6 & 4 & 2 & 5 & 3 & 6 \\
\end{array}
\right)   $$  

$$  \left(
\begin{array}{cccccc}
 1 & 1 & 1 & 1 & 1 & 1 \\
 2 & 2 & 4 & 6 & 3 & 5 \\
 3 & 6 & 3 & 5 & 2 & 4 \\
 4 & 5 & 2 & 4 & 6 & 3 \\
 5 & 4 & 6 & 3 & 5 & 2 \\
 6 & 3 & 5 & 2 & 4 & 6 \\
\end{array}
\right)   $$ 

\end{example}

\begin{remark} 
In summary, for order 6, we find a total of 12 non-isomorphic Alexander quandles, with 10 of those being multivariate quandles. This strongly suggests the need for an updated literature classification of Alexander quandles of various orders, which includes new examples coming from the multivariate kind.     
\end{remark}

\subsection{Link Colorings }

New isomorphism classes of quandles imply new coloring invariants for knots and links. As a concrete example let us consider the \emph{Whitehead link} (shown in Fig.~(\ref{figwhl1})), for which we show that the multivariate Alexander quandles of order 6 introduced above, generate new classes of non-trivial colorings. 
\begin{figure}[htbp]
\centering
\includegraphics[width=0.4\textwidth]{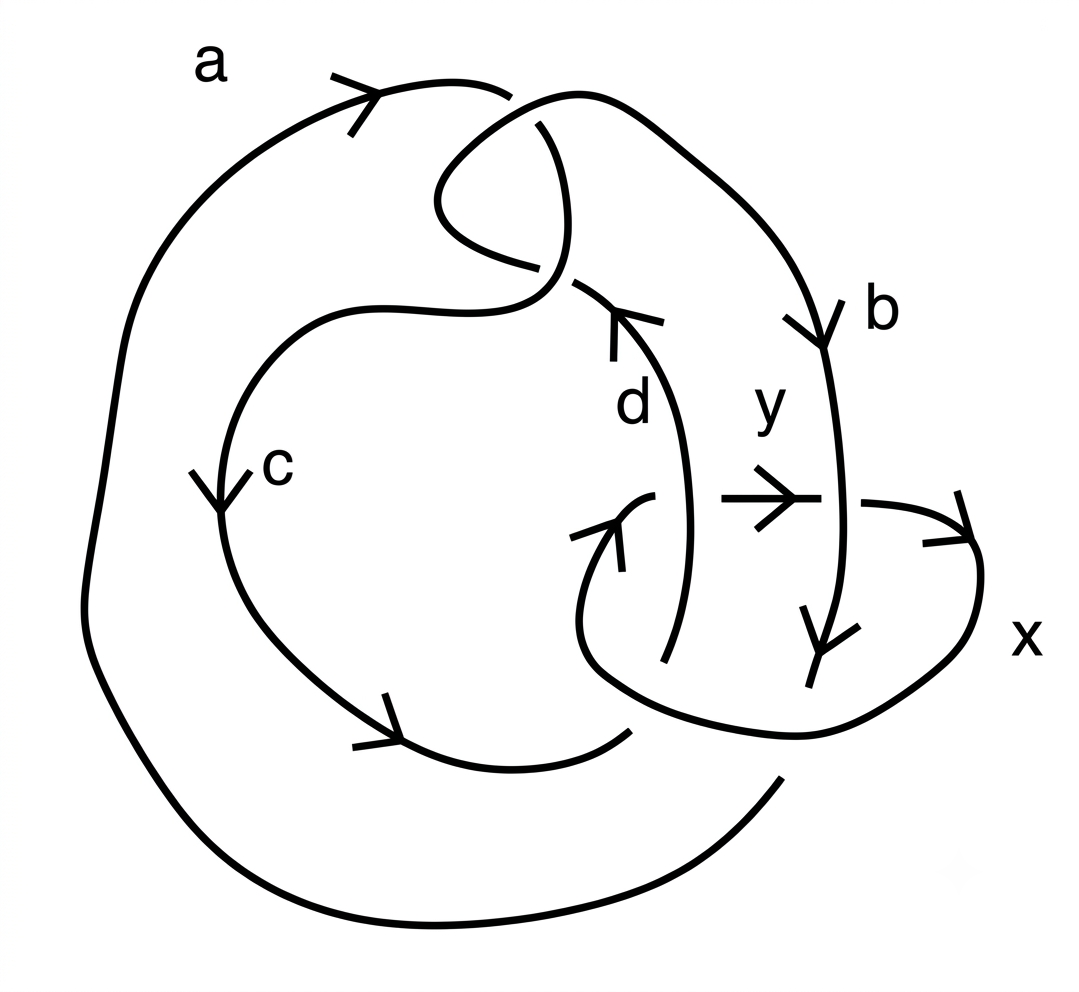}
\caption{The Whitehead link diagram with 6 crossings and 6 arcs. }
\label{figwhl1}
\end{figure}
For the arc-labeled link diagram in Fig.~(\ref{figwhl1}), we have the following quandle equations at the 6 crossings:
\begin{eqnarray*}
&a \;\bar{*}\; b \,=\, c&  \qquad  \,  d \; \bar{*} \; c \;=\, b   \qquad  \quad  c \;\bar{*}\; x \,=\, d \\ 
 &x \;\bar{*}\; d \,=\, y&  \qquad   y \,*\, b \,=\, x    \qquad    \quad  \!\!  b \,*\, x \,=\, a
\end{eqnarray*}
We now solve these equations for the colorings of the Whitehead link using the multivariate quandle of order 6 presented in Example~\ref{eg5.1.0}$(iii)$ (the one referring to solution $class \, (iii)$). Recall that this quandle is based on the following disjoint union  
$\mathbb{Z}_3 [t, \, t^{-1}]  / (t-1) \; \bigsqcup  \; \mathbb{Z}_3 [t, \, t^{-1}]  / (t-2)$  
and uses the operations given in eqs.~(\ref{op5}) and (\ref{op6}). 

We find 30 distinct coloring solutions for the 6 arcs of the Whitehead link. These are shown in Table~\ref{tabc1}, where each 6-tuple is a solution that consists of color labels numbered from 1 to 6, and ordered from left to right to denote colors of the 6 arcs $a, b, c, d, x, y$ respectively. For illustrative purposes, we also display these solutions using the barcodes shown in Fig.~(\ref{figwhl2}).

\begin{table}[htbp]
\centering 
\begin{tabular}{|c|} 
\hline   
$\begin{array}{c|c|c}
  \{1, 1, 1, 1, 1, 1\} & \{2, 2, 2, 2, 1, 1\} & \{3, 3, 3, 3, 1, 1\} \\   
  \{4, 4, 4, 4, 1, 1\} & \{5, 5, 5, 5, 1, 1\} & \{6, 6, 6, 6, 1, 1\} \\  
  \{1, 1, 1, 1, 2, 2\} & \{2, 2, 2, 2, 2, 2\} & \{3, 3, 3, 3, 2, 2\} \\   
  \{6, 4, 5, 6, 2, 3\} & \{4, 5, 6, 4, 2, 3\} & \{5, 6, 4, 5, 2, 3\} \\  
  \{1, 1, 1, 1, 3, 3\} & \{2, 2, 2, 2, 3, 3\} & \{3, 3, 3, 3, 3, 3\} \\
  \{5, 4, 6, 5, 3, 2\} & \{6, 5, 4, 6, 3, 2\} & \{4, 6, 5, 4, 3, 2\} \\
  \{1, 1, 1, 1, 4, 4\} & \{3, 2, 3, 2, 4, 5\} & \{2, 3, 2, 3, 4, 6\} \\
  \{4, 4, 4, 4, 4, 4\} & \{1, 1, 1, 1, 5, 5\} & \{3, 2, 3, 2, 5, 6\} \\
  \{2, 3, 2, 3, 5, 4\} & \{5, 5, 5, 5, 5, 5\} & \{1, 1, 1, 1, 6, 6\} \\
  \{3, 2, 3, 2, 6, 4\} & \{2, 3, 2, 3, 6, 5\} & \{6, 6, 6, 6, 6, 6\} \\
\end{array}$  \\ 
\hline
\end{tabular}
\caption{All 30 coloring solutions of the Whitehead link generated by the order 6 multivariate Alexander quandle of $class \, (iii)$ (Example~\ref{eg5.1.0}$(iii)$).}
\label{tabc1} 
\end{table}

Recall that the total number of coloring solutions, generated by a given quandle, is an invariant under Reidemeister moves (we have indeed checked that an alternate presentation of the Whitehead link with 5 crossings and 5 arcs  produces exactly 30 colorings of the 5 arcs, with this same quandle). Of the 30 colorings in Table~\ref{tabc1}, 18 are trivial solutions (either the whole link carries the same color, each component carries only 1 color), whereas we find 12 proper link colorings generated by this quandle. 

Now let us contrast the above colorings with those obtained by the univariate Alexander quandle of order 6, with $M = \mathbb{Z}_6 [t, \, t^{-1}]  / (t-5)$, for the same link diagram in Fig.~(\ref{figwhl1}). The solutions are shown in Fig.~(\ref{figwhl3}). This quandle only generates 12 colorings of the Whitehead link, and all 12 are trivial. The contrast with multivariate quandles is striking.

\begin{figure}[htbp]
\centering
\includegraphics[width=0.8\textwidth]{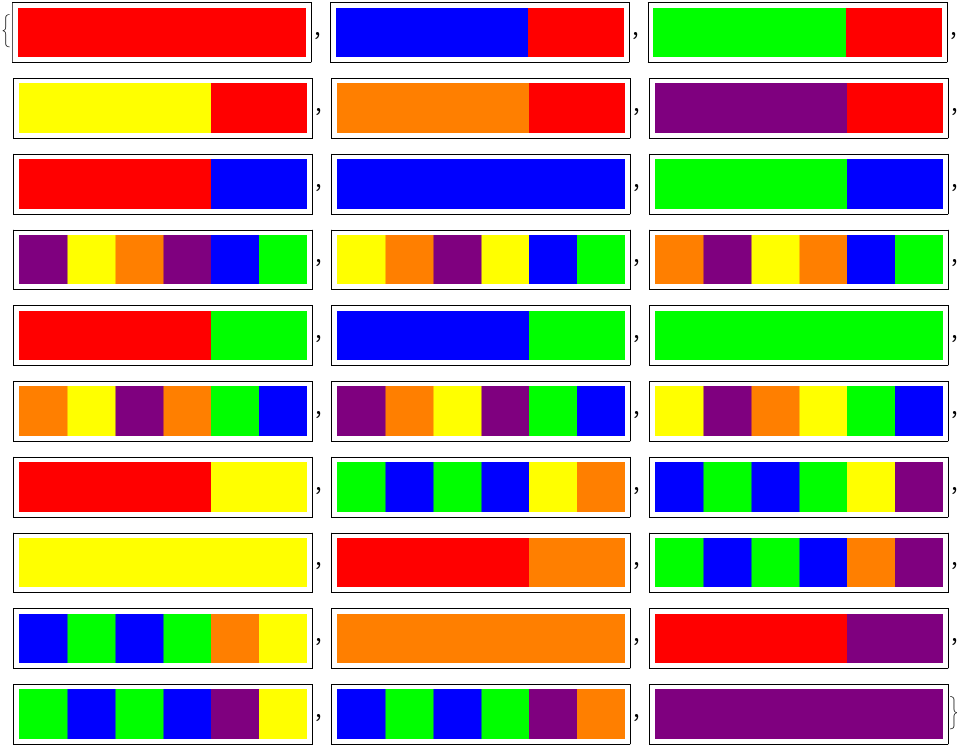}
\caption{Barcode representation of the Whitehead link colorings in Table~\ref{tabc1}. The color map is as follows: 1 $\to$ Red, 2 $\to$ Blue, 3 $\to$ Green, 4 $\to$ Yellow, 5 $\to$ Orange, 6 $\to$ Purple.}
\label{figwhl2}
\end{figure}

\begin{figure}[htbp!]
\centering
\includegraphics[width=0.8\textwidth]{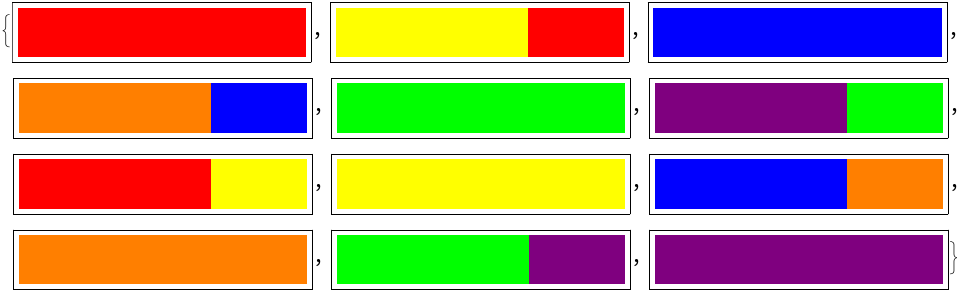}
\caption{Barcode representation (using the same color map as in Fig.~(\ref{figwhl2})) of the Whitehead link colorings (based on the diagram Fig.~(\ref{figwhl1})) generated by  the univariate Alexander quandle of order 6, with $M = \mathbb{Z}_6 [t, \, t^{-1}]  / (t-5)$. }
\label{figwhl3}
\end{figure}

\section{Conclusions and Discussion}  

To conclude, here we have introduced a new structural framework for  constructing (finite) multivariate Alexander quandles based on groupoid rings and algebroids. Our groupoids consist of a disjoint union of delooping groupoids of deck groups. Elements of this  groupoid are tagged by a label referring to the vertex group to which they belong. With this, we demonstrated how  groupoid rings become a natural extension to group rings used in classical knot topology. A manifestly categorical description of our groupoid ring was identified as a  $\Z$-algebroid. The two structure turn out to be different presentations of the same module theory. Using the latter, we were led to a  groupoid-based definitions of multivariate Alexander modules and multivariate Alexander quandles. The latter necessitated the category of elements extracted from multivariate modules. We then found three classes of multivariate quandle operations that followed from our definition. These satisfy the three quandle axioms, along with the constraint $t_{(a * b)} = t_a$. The latter also showed up in an alternate derivation we presented, based on Fox calculus with a generalized abelianization map. One of our operations  is similar to the operation proposed in \cite{traldi1},  but unlike  \cite{traldi1}, is independent of the link's presentation. The other two multivariate quandle operations are new quandle operations. With this, we were able to compute explicit instances of multivariate Alexander quandles, adding new examples to the existing literature classification. As a further application, we also demonstrated new link colorings for the Whitehead link.   
 
Multivariate Alexander quandles turn out to be composite structures, built out of quandles of lower orders. Hence, we get a quiver representation, with modules as objects, and quandle actions between modules as morphisms, thus providing constructions for composing several univariate Alexander quandles to obtain new multivariate ones. In this sense, multivariate Alexander quandles  follow naturally from the oidification of Alexander modules\footnote{The oidification of Alexander modules here, already follows from the construction of $\calG (n_1, \ldots, n_m)$. The additional oidification of rings to algebroids is   for the purpose of making definitions manifestly categorical.}.  

Compared to earlier definitions of multivariate Alexander quandles, the groupoid formulation we have proposed here very neatly distinguishes the algebraic machinery of multivariate Alexander quandles from any explicit dependence on the link's presentation. This is particularly advantageous for computing with multivariate quandle algebras directly, rather than via a knot or a link that they may refer to. The main applications  that become immediately amenable to this new framework are enumeration of  new classes of Alexander quandles, and new coloring invariants of links. We have computed explicit examples in this work, showing how the above work.  

Besides serving as a new machinery for computations, one may ask about the topological significance of a groupoid based framework. To answer that, firstly notice, that the groupoid $\calG (n_1, \ldots, n_m)$  we have introduced here, is not by itself a groupoid of deck groups (as objects of the groupoid). Rather, it is a disjoint union of delooping groupoids of deck groups $\bigsqcup_{i=1}^{m} B\Z_{n_i}$, where each $B\Z_{n_i}$ is the one-object groupoid whose automorphism group is $\Z_{n_i}$. Each vertex $i \in Ob({\calG})$ represents a "link component" that can be associated to a  covering space with deck group $\Z_{n_i}$. The delooping groupoid $B\Z_{n_i}$ is thus the categorical encoding of the deck group for component $i$.  $\calG$ then provides a synthetic framework that encodes deck groups in a purely algebraic manner, without invoking the covering space construction that typically gives rise to them. Furthermore, the groupoid ring $\Z[{\cal G}]$  generalizes the role that $\Z[t, t^{-1}]$  plays in classical Alexander theory, but now with multiple independent $t_i$-variables, each acting only on its own component. $\Z[{\cal G}]$ (equivalently, $\Zb[{\cal G}]$) thus provides a structural framework for defining modules over multiple deck groups simultaneously. The topological significance of this is that these structures could, in principle, be derived from a complicated covering space for the full link complement, but the groupoid framework serves as a useful device for organizing independent deck groups in a way that is self-contained and admits a generalization of Alexander quandles to multivariate quandles.

Let us now mention a few open questions following our investigations here. Firstly, the problem of classifying and enumerating finite Alexander quandles becomes open once again. Of course, one can computationally obtain Cayley tables for new multivariate Alexander quandles at low orders, but it would be desirable to have an algebraic (or geometric) way to estimate the number of Alexander quandles at each order. For instance, in recent work \cite{gonzalez2025representations}, Alexander quandles were associated to coherent sheaves constructed from character  varieties over $AGL_1({\mathbb C})$. It would be interesting to investigate whether some such geometric picture of multivariate Alexander quandles might help map the classification problem to an algebraic geometry counting problem.

Then, there is a deeper question of whether the groupoid and $\Z$-algebroid constructions here are indicative of a more general framework based on a   fundamental groupoid of the link complement. The utility of the fundamental groupoid has long been championed by R. Brown \cite{brown1987groups, brown2006topology}, as a geometric foundation for homotopy theory. In the context of link topology, our constructions here are suggestive of a basepoint-free formulation of covering space theory, which may possibly involve Brown's fundamental groupoid.  Additionally, in contrast to Milnor's link group \cite{milnor1954link}, which forgets all knotting data, a link groupoid, might presumably be a means to encode both, knotting and linking invariants.  This raises the interesting question of yet other unknown link invariants that may be  native to groupoid-based frameworks, such as the fundamental groupoid.

 Analogous to state sum invariants arising out of quandle cohomology \cite{carter2003quandle}, a natural extension of that would be to develop the cohomology of multivariate quandles and seek new invariants resulting from multivariate cocycle conditions. 

Finally, notice that the (finite) multivariate Alexander quandles we have constructed here, are all composed of quandles of lower orders, along the  diagonal blocks. The quiver itself is indicative of such a compositional structure. This raises two interesting questions: (i) Do there also exist "irreducible" finite multivariate Alexander quandles, which do not result from compositions of lower order quandles? and, (ii) In the case of infinite quandles, does a counterpart to such a compositional structure exist? The latter would be directly relevant to the fundamental link quandle.



\bibliographystyle{eptcs}
\bibliography{hottrefs.bib}

\end{document}